\documentclass[10pt]{amsart}

\usepackage{amsmath, amssymb, graphicx}
\usepackage{pst-all}
\usepackage{xspace}

\newcommand{\pup}[1]{\textup{(}#1\textup{)}}

\allowdisplaybreaks

\theoremstyle{theorem}

\newtheorem{coro}[equation]{Corollary}

\newtheorem{lemm}[equation]{Lemma}
\newtheorem{prop}[equation]{Proposition}

\newtheorem*{propA}{Proposition~A {\rm(Prop.\,\ref{P:PropA})}}
\newtheorem*{propB}{Proposition~B {\rm(Prop.\,\ref{P:PropB})}}
\newtheorem*{propC}{Proposition~C {\rm(Prop.\,\ref{P:PropC})}}
\theoremstyle{definition}
\newtheorem{defi}[equation]{Definition}
\newtheorem{exam}[equation]{Example}
\newtheorem{nota}[equation]{Notation}
\newtheorem{rema}[equation]{Remark}

\psset{unit=1mm}
\psset{fillcolor=white}
\psset{dotsep=1.5pt}
\psset{dash=1.5pt 1.5pt}
\psset{linewidth=0.8pt}
\psset{arrowsize=3.5pt}
\psset{doublesep=1pt}
\psset{coilwidth=1.2mm}
\psset{coilheight=2}
\psset{coilaspect=20}
\psset{coilarm=2}
\newpsstyle{double}{linewidth=0.5pt,doubleline=true,arrowsize=5pt}
\newpsstyle{doubleexist}{linewidth=0.5pt,doubleline=true,arrowsize=5pt,linestyle=dashed}
\newpsstyle{etc}{linestyle=dotted}
\newpsstyle{exist}{linestyle=dashed}
\newpsstyle{thin}{linewidth=0.5pt}
\newpsstyle{thinexist}{linewidth=0.5pt, linestyle=dashed}
\newpsstyle{back}{linewidth=3mm,linearc=1}
\def\AArrow(#1,#2){\ncline[nodesep=0.3mm, linewidth=0.8pt, border=1.2pt]{->}{#1}{#2}}
\def\BArrow(#1,#2){\ncline[linecolor=blue, linewidth=1.2pt, linestyle=dotted, nodesep=0.3mm, border=1.2pt]{->}{#1}{#2}}
\def\CArrow(#1,#2){\ncline[linecolor=red, nodesep=0.3mm, linewidth=0.8pt, border=1.2pt]{->}{#1}{#2}}
\definecolor{color1}{rgb}{.88,.88,.88}
\definecolor{color2}{rgb}{0.85,.92,1}
\definecolor{color20}{rgb}{0.9,.95,1}
\definecolor{color21}{rgb}{0.8,.9,1}
\definecolor{color22}{rgb}{0.75,.85,1}
\definecolor{color3}{rgb}{0,0,1}
\definecolor{color2}{rgb}{1,.85,0.85}
\definecolor{color20}{rgb}{.88,.88,.88}
\definecolor{color21}{rgb}{1,.8,.8}
\definecolor{color22}{rgb}{1,.75,.75}
\definecolor{color3}{rgb}{1,0,0}

\def\WPoint(#1,#2,#3){\cnode[style=thin,fillcolor=white,fillstyle=solid](#1,#2){0.7}{#3}}
\def\AArrow(#1,#2){\ncline[nodesep=1mm, linewidth=0.8pt, border=2pt]{->}{#1}{#2}}
\def\PArrow(#1,#2){\ncline[linestyle=dashed, nodesep=0.3mm, linewidth=0.8pt, border=0.8pt]{->}{#1}{#2}}

\numberwithin{equation}{section}

\newcounter{ITEM}
\newcommand\ITEM[1]{\setcounter{ITEM}{#1}\leavevmode\hbox{\rm(\roman{ITEM})}}

\renewcommand\aa{a}
\renewcommand\AA{A}
\newcommand\aav{\U\aa}
\newcommand\act{\mathbin{\scriptscriptstyle\bullet}}
\renewcommand\and{\text{and}}

\newcommand\bb{b}
\newcommand\BB{B}
\newcommand\bbv{\U\bb}
\newcommand\can{\iota}
\newcommand\card{\mathtt\#}
\newcommand{\Cat}[1]{\operatorname{Cat}(#1)}
\newcommand\cc{c}
\newcommand\CC{C}
\newcommand\CCC{\mathcal{C}}
\newcommand\ccv{\U\cc}

\newcommand\ConjA{\mathbf{A}}
\newcommand\ConjB{\mathbf{B}}
\newcommand\dd{d}
\newcommand\DD{D}
\newcommand\ddv{\U\dd}
\renewcommand\dh[1]{\Vert#1\Vert}
\renewcommand\div{<}

\newcommand\dive{\le}

\newcommand\divet{\mathrel{\widetilde{\VR(1.8,0)\smash{\dive}}}}
\newcommand\divt{\mathrel{\widetilde{\VR(1.8,0)\smash{<}}}}
\newcommand\ee{e}

\newcommand\EG[1]{\mathrm{U_{gp}}(#1)}
\newcommand\FG[1]{\mathrm{F_{\!gp}}(#1)}
\newcommand\ef{\varnothing}
\newcommand\eps{\varepsilon}

\newcommand\ff{f}
\newcommand\FF{F}
\newcommand\FRb[1]{\mathcal{F}^{\scriptscriptstyle\pm}_{\HS{-0.6}#1}}
\renewcommand\gcd{\wedge}
\newcommand\gcdt{\mathbin{\widetilde\wedge}}
\renewcommand\ge{\geqslant}

\renewcommand\gg{g}

\newcommand\HS[1]{\hspace{#1ex}}

\newcommand\ie{i.e.}
\newcommand\ii{i}
\newcommand\II{I}
\newcommand\IM[1]{\mathrm{Int}(#1)}
\newcommand\INF[2]{#1^{\le#2}}
\newcommand\inv{^{-1}}

\newcommand\INV[1]{\overline{#1}}
\newcommand\jj{j}
\newcommand\JJ{J}
\newcommand\kk{k}
\newcommand\KK{K}
\newcommand\lcm{\vee}
\newcommand\lcmt{\mathbin{\widetilde\lcm}}
\renewcommand\le{\leqslant}

\newcommand\MM{M}
\newcommand\MMA{M_{\HS{-0.2}A}}
\newcommand\MMAA[1]{M_{\HS{-0.2}A, #1}}
\newcommand\MMB{M_{\HS{-0.2}B}}
\newcommand\MMCC[1]{M_{\HS{-0.2}C, #1}}
\newcommand\MMD{M_D}

\newcommand\MON[2]{\langle#1\mid#2\rangle^{\!+}}
\newcommand\mult{>}
\newcommand\multe{\ge}
\newcommand\multt{\mathrel{\widetilde{\VR(1.8,0)\smash{>}}}}
\newcommand\multet{\mathrel{\widetilde{\VR(1.8,0)\smash{\ge}}}}
\newcommand\NF{\textsc{nf}}
\newcommand\nn{n}
\newcommand\NN{N}

\newcommand\One[1]{\underline1_{#1}}
\newcommand\one{\underline1}
\newcommand\opp{\cdot}
\newcommand\pdots{\HS{0.2}{\cdot}{\cdot}{\cdot}\HS{0.2}}
\newcommand\piece{\mathrel{\triangleleft}}
\newcommand\piecet{\mathrel{\triangleright}}
\newcommand\pp{p}
\newcommand\PP{P}
\newcommand\PPA{P_{\HS{-0.3}A}}
\newcommand\PPAA[1]{P_{\HS{-0.3}A, #1}}
\newcommand\PPB{P_{\HS{-0.1}B}}
\newcommand\PPCC[1]{P_{\HS{-0.1}C, #1}}

\newcommand\PROD[2]{\mathrel{{#1}{\vert}{#2}}}

\newcommand\Pw{\mathfrak{P}}
\newcommand\qq{q}
\newcommand\quand{\quad\text{and}\quad}

\newcommand\rd{\Rightarrow}
\newcommand\rds{\rd^{\HS{-0.3}*}}
\renewcommand\red{\mathrm{red}}
\newcommand\Red[1]{R_{#1}}
\newcommand\resp{resp.,\ }
\newcommand\rr{r}
\newcommand\RR{R}

\newcommand\sdots{ / \pdots / }
\newcommand\Sim[1]{\mathcal{S}_{\HS{-0.2}#1}}
\newcommand\simeqb{\simeq^{\HS{-0.2}\scriptscriptstyle\pm}}

\newcommand\Simm[1]{\mathcal{S}^{\scriptscriptstyle\mathrm{min}}_{\HS{-0.2}#1}}
\newcommand\sr[1]{\partial_0#1}
\renewcommand\ss{s}
\renewcommand\SS{S}
\newcommand\SUP[2]{#1^{\ge#2}}
\newcommand\tg[1]{\partial_1#1}
\renewcommand\tt{t}

\newcommand\tta{\mathtt{a}}
\newcommand\ttb{\mathtt{b}}
\newcommand\ttc{\mathtt{c}}

\let\U=\underline
\newcommand{\Um}[1]{\operatorname{U_{mon}}(#1)}
\newcommand\uu{u}
\def\VR(#1,#2){\vrule width0pt height#1mm depth#2mm}

\newcommand\vv{v}
\newcommand\wdots{, ...\HS{0.2},}

\newcommand\xx{x}
\newcommand\XX{X}
\newcommand\xxv{\U{x}}
\newcommand\yy{y}

\newcommand\zz{z}

\title[Multifraction reduction III]{Multifraction reduction III: The case of interval monoids}

\author{Patrick Dehornoy}

\address{Laboratoire de Math\'ematiques Nicolas Oresme, CNRS UMR 6139, Universit\'e de Caen, 14032 Caen cedex, France, and Institut Universitaire de France}
\email{patrick.dehornoy@unicaen.fr}
\urladdr{www.math.unicaen.fr/\~{}dehornoy}

\author{Friedrich Wehrung}

\address{Laboratoire de Math\'ematiques Nicolas Oresme, CNRS UMR 6139, Universit\'e de Caen, 14032 Caen cedex, France}
\email{friedrich.wehrung01@unicaen.fr}
\urladdr{www.math.unicaen.fr/\~{}wehrung}


\begin{document}

\maketitle

\begin{abstract}
We investigate gcd-monoids, which are cancellative monoids in which any two elements admit a left and a right gcd, and the associated reduction of multifractions (arXiv:1606.08991 and 1606.08995), a general approach to the word problem for the enveloping group. Here we consider the particular case of interval monoids associated with finite posets. In this way, we construct gcd-monoids, in which reduction of multifractions has prescribed properties not yet known to be compatible: semi-convergence of reduction without convergence, semi-convergence up to some level but not beyond, non-embeddability into the enveloping group (a strong negation of semi-convergence).
\end{abstract}

\section{Introduction}\label{S:Intro}

Reduction of multifractions~\cite{Dit, Diu} is a new approach to the word problem for Artin-Tits groups and, more generally, for groups that are enveloping groups of monoids in which the divisibility relations have weak lattice properties (``gcd-monoids''). It is based on a reduction system that extends the usual free reduction for free groups. 

It is proved in~\cite{Dit} that, when the ground monoid~$\MM$ satisfies an additional condition called the $3$-Ore condition, then the associated rewrite system (``reduction '') is convergent, leading to a solution of the word problem for the enveloping group of the monoid whenever convenient finiteness assumptions are met. Next, it is proved in~\cite{Diu} that, even when the $3$-Ore condition fails, typically in the case of Artin-Tits monoids that are not of FC~type~\cite{Alt, GoP2}, most consequences of the convergence of reduction, in particular a solution for the word problem of the enveloping group, can still be obtained when a weak form of convergence called \emph{semi-convergence} is satisfied. It is conjectured that reduction is semi-convergent for every Artin-Tits monoid~$\MM$. Computer experiments and partial results support this conjecture, but it remains so far open. The connection with the word problem for Artin-Tits groups, an open question in the general case, makes investigating the semi-convergence of reduction both natural and important.

 After~\cite{Diu}, little was known about the strength of semi-converg\-ence of reduction. In particular, the following questions were left open: Is semi-convergence strictly weaker than convergence? In the opposite direction, could it be that reduction is semi-convergent for every gcd-monoid~$\MM$? The aim of the current paper is to answer the above two questions, in a strong sense. A sequence of approximations of semi-convergence called $\nn$-semi-convergence, $\nn$ even $\ge 2$, is introduced in~\cite{Diu}. Here we prove

\begin{propA}
There exists an explicit gcd-monoid~$\MMA$ for which reduction is semi-convergent but not convergent.
\end{propA}

\begin{propB}
There exists an explicit gcd-monoid~$\MMB$ for which reduction is not $2$-semi-convergent, which amounts to saying that the monoid~$\MMB$ does not embed into the enveloping group of~$\MMB$.
\end{propB}

\begin{propC}
For every~$\nn \ge 4$ even, there exists an explicit gcd-monoid~$\MMCC\nn$ for which reduction is $\pp$-semi-convergent for~$\pp < \nn$ but not $\nn$-semi-convergent. For $\nn = 4$, this amounts to saying that there exists a gcd-monoid~$\MMCC4$ that embeds into the enveloping group of~$\MMCC4$ but some fraction in this group has more than one irreducible expression.
\end{propC}

In order to prove the above results, we appeal to \emph{interval monoids} of (finite) posets (\ie, partially ordered sets), a particular case of the general family of category monoids (\ie, universal monoids of categories) investigated in~\cite{Weh}. In this approach, to every poset~$P$ we associate its interval monoid~$\IM\PP$. These monoids prove to be very convenient for investigating the current questions, as a number of properties of~$\IM\PP$ and the derived reduction system boil down to combinatorial features involving~$\PP$, typically the circuits in its Hasse diagram. The point for establishing the above results is then to construct posets with specified properties. 

One of the consequences of the current results, in particular the negative result of Proposition~C, is that (as could be expected) a possible proof of semi-convergence of reduction for all Artin-Tits monoids will require specific ingredients. However, as explained at the end of the paper, the method used to prove Proposition~A may suggest some approaches.

The organization of the paper, which is essentially self-contained, is as follows. After some prerequisites about gcd-monoids and the associated reduction system in Section~\ref{S:Red}, we describe interval monoids~$\IM\PP$ in Section~\ref{S:OrdMon}. In Section~\ref{S:Embed}, we consider the specific case of $2$-semi-convergence, that is, the embeddability of the monoid into its enveloping group. Next, we develop in Section~\ref{S:NCInt} a general method for possibly establishing semi-convergence by restricting to particular multifractions and study this method in the specific case of interval monoids, establishing a simple homotopical criterion involving the ground poset. Positive and negative examples are then described in Section~\ref{S:Ex}. Finally, we briefly discuss in Section~\ref{S:Ext} the possibility of extending the approach to gcd-monoids that are not interval monoids, in particular Artin-Tits monoids.

\section{Multifraction reduction}\label{S:Red}

Here we recall some background about the enveloping group of a monoid, about gcd-monoids, and about the reduction system that is our main subject of investigation. Most proofs are omitted and can be found in~\cite{Dit, Diu}.

\subsection{Multifractions}\label{SS:Multifrac}

We denote by~$\EG\MM$ the \emph{enveloping group} (called \emph{universal group} in~\cite{Weh}) of a monoid~$\MM$, which is the (unique) group with the universal property that every morphism from~$\MM$ to a group uniquely factors through~$\EG\MM$. The elements of~$\EG\MM$ can be represented using finite sequences of elements of~$\MM$, here called mutifractions. 

\begin{defi}\label{D:Multifrac}
Let~$\MM$ be a monoid, and let $\INV\MM=\{\INV{x}\mid x\in\MM\}$ be a disjoint copy of~$\MM$. For $\nn \ge 1$, an \emph{$\nn$-multifraction} on~$\MM$ is a finite sequence $(\aa_1 \wdots \aa_\nn)$ with entries alternatively in~$\MM$ and~$\INV\MM$ (starting with either). The set of all multifractions completed with the empty sequence~$\ef$ is denoted by~$\FRb\MM$. Multifractions are multiplied using the rule
\begin{equation*}\label{E:SignedProd}
(\aa_1 \wdots \aa_\nn) \opp (\bb_1 \wdots \bb_\pp) = 
\begin{cases}
(\aa_1 \wdots \aa_{\nn-1}, \aa_\nn \bb_1, \bb_2 \wdots \bb_\pp)
&\quad\text{for $\aa_\nn$ and $\bb_1$ in~$\MM$,}\\
(\aa_1 \wdots \aa_{\nn-1}, \bb_1\aa_\nn, \bb_2 \wdots \bb_\pp)
&\quad\text{for $\aa_\nn$ and $\bb_1$ in~$\INV\MM$,}\\
(\aa_1 \wdots \aa_\nn , \bb_1 \wdots \bb_\pp)
&\quad\text{otherwise,}\\
\end{cases}
\end{equation*}
extended with $\aav \opp \ef = \ef \opp \aav = \aav$ for every~$\aav$\end{defi}

We use $\aav$, $\bbv$, \dots, as generic symbols for multifractions, and denote by~$\aa_\ii$ the $\ii$th entry of~$\aav$ counted from~$1$.
The number of entries in a multifraction~$\aav$ is called its \emph{depth}, written~$\dh\aav$. For every~$\aa$ in~$\MM$, we identify~$\aa$ with the $1$-multifraction~$(\aa)$. The following convention is helpful to view multifractions as an iteration of usual fractions:

\begin{nota}\label{N:Frac}
For $\aa_1 \wdots \aa_\nn$ in~$\MM$, we put
\begin{equation}\label{E:Frac}
\aa_1 \sdots \aa_\nn:= (\aa_1, \INV{\aa_2}, \aa_3, \INV{\aa_4}, \dots) \quand 
/ \aa_1 \sdots \aa_\nn:= (\INV{\aa_1}, \aa_2, \INV{\aa_3}, \aa_4, \dots);
\end{equation}
We say that $\ii$ is \emph{positive} (\resp \emph{negative}) \emph{in}~$\aav$ if $\aa_\ii$ (\resp $\INV{\aa_\ii}$) occurs in~\eqref{E:Frac}.
\end{nota}

The following is then easy:

\begin{prop}\cite[Prop.\,2.5]{Diu}\label{P:EnvGroup}
\ITEM1 The set~$\FRb\MM$ equipped with~$\opp$ is a monoid generated by $\MM \cup \INV\MM$.

\ITEM2 Let $\simeqb$ be the congruence on~$\FRb\MM$ generated by $(1, \ef)$ and the pairs $(\aa / \aa, \ef)$ and $(/ \aa / \aa, \ef)$ with~$\aa$ in~$\MM$, and, for~$\aav$ in~$\FRb\MM$, let~$\can(\aav)$ be the $\simeqb$-class of~$\aav$. Then the group~$\EG\MM$ is \pup{isomorphic to} $\FRb\MM{/}{\simeqb}$ and, for every~$\aav$, we have
\begin{equation}\label{E:Eval}
\can(\aav) = \can(\aa_1) \, \can(\aa_2)\inv \, \can(\aa_3) \, \can(\aa_4)\inv \pdots. 
\end{equation}
\end{prop}

Hereafter, we identify $\EG\MM$ with $\FRb\MM{/}{\simeqb}$. A multifraction is called \emph{positive} if its first entry is positive, \emph{trivial} if all its entries are trivial (\ie, equal to~$1$), and \emph{unital} if it represents~$1$ in the group~$\EG\MM$. A trivial multifraction is unital, but, of course, the converse is not true.

\subsection{Gcd-monoids}\label{SS:Gcd}

The reduction process we shall consider makes sense when the ground monoid is a gcd-monoid. Here we recall the basic definitions, referring to~\cite{Dit} for more details.

If $\MM$ is a monoid, then, for~$\aa, \bb$ in~$\MM$, we say that $\aa$ \emph{left divides}~$\bb$ or, equivalently, that $\bb$ is a \emph{right multiple} of~$\aa$, and write~$\aa \dive \bb$, if $\aa\xx = \bb$ holds for some~$\xx$ in~$\MM$. If $\MM$ is cancellative and $1$ is the only invertible element in~$\MM$, the left divisibility relation is a partial ordering on~$\MM$. In this case, the greatest common lower bound of two elements~$\aa, \bb$ with respect to~$\dive$, if it exists, is called their \emph{left gcd}, denoted by~$\aa \gcd \bb$, whereas their least common upper bound, if it exists, is called their \emph{right lcm}, denoted by~$\aa \lcm \bb$. 

Symmetrically, we say that $\aa$ \emph{right divides}~$\bb$ or, equivalently, that $\bb$ is a \emph{left multiple} of~$\aa$, written~$\aa \divet \bb$, if $\xx\aa = \bb$ holds for some~$\xx$. Under the same hypotheses, $\divet$ is a partial ordering on~$\MM$, with the derived right gcd~$\gcdt$ and left lcm~$\lcmt$.

\begin{defi}\label{D:GcdMon}
We say that $\MM$ is a \emph{gcd-monoid} if $\MM$ is a cancellative monoid, $1$ is the only invertible element in~$\MM$, and any two elements of~$\MM$ admit a left and a right gcd.
\end{defi}

Typical examples are Artin-Tits monoids, that is, the monoids $\MON\SS\RR$, where~$\RR$ contains at most one relation $\ss\pdots = \tt\pdots$ 
for each pair of generators~$\ss, \tt$ in~$\SS$ and, if so, it has the form $\ss \tt \ss \tt\pdots = \tt \ss \tt \ss\pdots$, both sides of the same length~\cite{BrS, Dlg}.

Standard arguments (see for instance~\cite[Lemma~2.13]{Dit}) show that, if $\MM$ is a gcd-monoid, then $\MM$ admits conditional right and left lcms, that is, any two elements of~$\MM$ that admit a common right multiple (\resp left multiple) admit a right lcm (\resp a left lcm) and that, conversely, every cancellative monoid with no nontrivial invertible element that admits conditional left and right lcms is a gcd-monoid.

A cancellative monoid~$\MM$ with no nontrivial invertible element is called \emph{noetherian} if the proper left and right divisibility relations~$\div, \divt$ are both well-founded, that is, they admit no infinite descending sequence. Note that every presented monoid~$\MON\SS\RR$ where $\RR$ consists of homogeneous relations, that is, of relations $\uu = \vv$ with~$\uu$, $\vv$ of the same length, is noetherian, and even \emph{strongly noetherian}, meaning that there exists a map~$\lambda$ from~$\MM$ to nonnegative integers satisfying $\lambda(\aa\bb) \ge \lambda(\aa) + \lambda(\bb)$ for all~$\aa, \bb$ in~$\MM$ and $\lambda(\aa) > 0$ for $\aa \not= 1$ \cite[Sec.\,II.2.4]{Dir}.

Noetherianity implies the existence of \emph{atoms}, namely elements that cannot be expressed as the product of two non-invertible elements. A noetherian cancellative monoid with no nontrivial invertible element is always generated by its atoms~\cite[Lemma~2.29]{Dit}.

\subsection{Reduction of multifractions}\label{SS:Reduction}

Owing to Proposition~\ref{P:EnvGroup}, studying the enveloping group~$\EG\MM$ of a monoid~$\MM$ amounts to understanding the congruence~$\simeqb$ on the multifraction monoid~$\FRb\MM$. To do this, we introduced in~\cite{Dit, Diu} a family of rewrite rules~$\Red{\ii, \xx}$ acting on multifractions. 

\begin{defi}\label{D:Redpm}
Assume that $\MM$ is a gcd-monoid. For $\aav$, $\bbv$ in~$\FRb\MM$, and for $\ii \ge 1$ and $\xx \in \MM$, we declare $\bbv = \aav \act \Red{\ii, \xx}$ if we have $\dh\bbv = \dh\aav$, $\bb_\kk = \aa_\kk$ for $\kk \not= \ii - 1, \ii, \ii + 1$, and there exists~$\xx'$ (necessarily unique) satisfying
$$\begin{array}{lccc}
\text{for $\ii\ge 2$ positive in~$\aav$:\quad}
&\bb_{\ii-1} = \xx' \aa_{\ii-1}, 
&\bb_\ii \xx = \xx' \aa_\ii = \xx \lcmt \aa_\ii, 
&\bb_{\ii+1} \xx = \aa_{\ii+1},\\
\text{for $\ii \ge 2$ negative in~$\aav$:}
&\smash{\bb_{\ii-1} = \aa_{\ii-1} \xx'}, 
&\smash{\xx \bb_\ii = \aa_\ii \xx' = \xx \lcm \aa_\ii}, 
&\xx \bb_{\ii+1} = \aa_{\ii+1},\\
\text{for $\ii = 1$ positive in~$\aav$:}
&&\bb_\ii \xx = \aa_\ii, 
&\bb_{\ii+1} \xx = \aa_{\ii+1},\\
\text{for $\ii = 1$ negative in~$\aav$:}
&&\xx \bb_\ii = \aa_\ii, 
&\xx \bb_{\ii+1} = \aa_{\ii+1}.
\end{array}$$
We write $\aav \rd \bbv$ if $\aav \act \Red{\ii, \xx}$ holds for some~$\ii$ and some $\xx \not= 1$, and use $\rds$ for the reflexive--transitive closure of~$\rd$. The family of all rules~$\Red{\ii, \xx}$ is called \emph{reduction} (for the monoid~$\MM$). 
\end{defi}

A multifraction~$\aav$ is called \emph{reducible} if at least one rule~$\Red{\ii, \xx}$ with $\xx \not= 1$ applies to~$\aav$, and irreducible otherwise.

Reduction as defined above extends free reduction (deletion of factors~$\xx\inv \xx$ or~$\xx \xx\inv$): applying~$\Red{\ii, \xx}$ to a multifraction~$\aav$ consists in
pushing the factor~$\xx$ from the $(\ii + 1)$st level to the $(\ii - 1)$st level, using the lcm operation to cross the entry~$\aa_\ii$. This is illustrated in Figure~\ref{F:Red}, where the arrows correspond to the elements of the monoid (as if they were morphisms of a category), with concatenation corresponding to multiplication and squares to equalities.

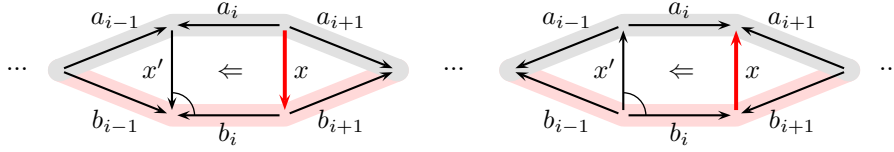
\begin{figure}[htb]
\begin{picture}(105,17)(0,-2)
\psset{nodesep=0.7mm}
\put(-7,6){$...$}
\psline[style=back,linecolor=color2]{c-c}(0,6)(15,0)(30,0)(45,6)
\psline[style=back,linecolor=color1]{c-c}(0,6)(15,12)(30,12)(45,6)
\pcline{->}(0,6)(15,12)\taput{$\aa_{\ii - 1}$}
\pcline{<-}(15,12)(30,12)\taput{$\aa_\ii$}
\pcline{->}(30,12)(45,6)\taput{$\aa_{\ii + 1}$}
\pcline{->}(0,6)(15,0)\tbput{$\bb_{\ii - 1}$}
\pcline{<-}(15,0)(30,0)\tbput{$\bb_\ii$}
\pcline{->}(30,0)(45,6)\tbput{$\bb_{\ii + 1}$}
\pcline[linewidth=1.5pt,linecolor=color3,arrowsize=1.5mm]{->}(30,12)(30,0)\trput{$\xx$}
\pcline{->}(15,12)(15,0)\tlput{$\xx'$}
\psarc[style=thin](15,0){3}{0}{90}
\put(21,5){$\Leftarrow$}
\put(51,6){$...$}
\psline[style=back,linecolor=color2]{c-c}(60,6)(75,0)(90,0)(105,6)
\psline[style=back,linecolor=color1]{c-c}(60,6)(75,12)(90,12)(105,6)
\pcline{<-}(60,6)(75,12)\taput{$\aa_{\ii - 1}$}
\pcline{->}(75,12)(90,12)\taput{$\aa_\ii$}
\pcline{<-}(90,12)(105,6)\taput{$\aa_{\ii + 1}$}
\pcline{<-}(60,6)(75,0)\tbput{$\bb_{\ii - 1}$}
\pcline{->}(75,0)(90,0)\tbput{$\bb_\ii$}
\pcline{<-}(90,0)(105,6)\tbput{$\bb_{\ii + 1}$}
\pcline[linewidth=1.5pt,linecolor=color3,arrowsize=1.5mm]{<-}(90,12)(90,0)\trput{$\xx$}
\pcline{<-}(75,12)(75,0)\tlput{$\xx'$}
\psarc[style=thin](75,0){3}{0}{90}
\put(81,5){$\Leftarrow$}
\put(109,6){$...$}
\end{picture}
\caption{\sf The reduction rule~$\Red{\ii, \xx}$: starting from~$\aav$ (the grey path), we extract $\xx$ from~$\aa_{\ii + 1}$, push it through~$\aa_\ii$ by taking the lcm of~$\xx$ and~$\aa_\ii$ (indicated by the small curved arc), and incorporate the remainder~$\xx'$ in~$\aa_{\ii - 1}$ to obtain~$\bbv = \aav \act \Red{\ii, \xx}$ (the colored path). The left hand side diagram corresponds to the case of $\ii$ negative in~$\aav$, the right hand side one to $\ii$ positive in~$\aav$, with opposite orientations of the arrows.}
\label{F:Red}
\end{figure}

\begin{exam}\label{X:Red}
If $\MM$ is a free commutative monoid, then every sequence of reductions starting from an arbitrary multifraction~$\aav$ leads in finitely many steps to an irreducible multifraction, namely one of the form $\bb_1 / \bb_2 / 1 \sdots 1$ where~$\bb_1$ and~$\bb_2$ share no letter.
\end{exam}

The following result gathers the basic properties of reduction needed for the current paper. We refer to~\cite{Dit, Diu} for the proofs. We use $\One\pp$ for $1 \sdots 1$, $\pp$~terms, abbreviated in~$\one$ in case~$\pp$ is not needed.

\begin{lemm}\cite{Dit, Diu}\label{L:Basic}
\ITEM1 For every gcd-monoid~$\MM$, the relation $\rds$ is included in~$\simeqb$ and it is compatible with multiplication on~$\FRb\MM$.

\ITEM2 If, moreover, $\MM$ is noetherian, then reduction is terminating for~$\MM$: every sequence of reductions leads in finitely many steps to an irreducible multifraction.
\end{lemm}

It is proved in~\cite{Dit} and~\cite{Diu} that reduction is convergent for~$\MM$ (meaning that, for every multifraction~$\aav$, there exists exactly one irreducible multifraction~$\red(\aav)$ to which $\aav$ reduces) if and only if the ground monoid~$\MM$ satisfies the $3$-Ore condition, namely that any three elements of~$\MM$ that pairwise admit common right (\resp left) multiples admit a global common right (\resp left) multiple. In this case, one obtains a full control of the congruence~$\simeqb$ and, from there, of the enveloping group~$\EG\MM$:
two multifractions~$\aav, \bbv$ with $\dh\aav \ge \dh\bbv$ are $\simeqb$-equivalent if and only if $\red(\aav) = \red(\bbv) \opp \one$ holds. Then, under convenient finiteness assumptions ensuring the decidability of the relation~$\rds$, one obtains a solution to the word problem for the group~$\EG\MM$. 

In many cases, for instance in the case of any Artin--Tits monoid that is not of FC~type~\cite{Alt, GoP2}, reduction is not convergent for~$\MM$, and there seems to be little hope to amend it, typically by adding new rules, so as to obtain a convergent system. However, some weak forms of convergence might be satisfied in more cases. If reduction is convergent, then
\begin{equation}\label{E:NC0}
\text{A multifraction~$\aav$ is unital if and only if $\aav \rds \one$ holds,}
\end{equation}
and it is shown in~\cite{Diu} that most of the consequences of the convergence of reduction, in particular the decidability of the word problem, already follow from~\eqref{E:NC0}. Moreover, all known examples contradicting convergence of reduction fail to contradict~\eqref{E:NC0}. This makes~\eqref{E:NC0} worth of investigation.

\begin{defi}\cite{Diu}
If $\MM$ is a gcd-monoid, we say that reduction is \emph{semi-convergent} for~$\MM$ if~\eqref{E:NC0} holds for every multifraction~$\aav$ on~$\MM$.
\end{defi}

This is the property we shall investigate in the rest of this paper. It will be convenient to start from the following slight variant. 

\begin{lemm}\label{L:NC}
If $\MM$ is a noetherian gcd-monoid, then reduction is semi-convergent for~$\MM$ if and only if
\begin{equation}\label{E:NC}
\parbox{110mm}{
Every unital multifraction on~$\MM$ is either trivial or reducible.}
\end{equation}
\end{lemm}

\begin{proof}
Assume that $\aav$ is unital, that is, $\aav$ represents~$1$ in~$\EG\MM$. Then \eqref{E:NC0} implies $\aav \rds \one$ so, by definition, $\aav$ is either trivial or reducible, and \eqref{E:NC0} implies~\eqref{E:NC}. For the other direction, the assumption that $\MM$ is noetherian implies that reduction for~$\MM$ is terminating, so there exists~$\bbv$ irreducible satisfying~$\aav \rds \bbv$. Then \eqref{E:NC} implies $\bbv = \one$. So \eqref{E:NC} implies~\eqref{E:NC0}.
\end{proof}

By definition, the reduction rules~$\Red{\ii, \xx}$ preserve the depth of multifractions and, therefore, it makes sense to consider the specialization of reduction to $\nn$-multifractions.

\begin{defi}\label{D:WC}
For a gcd-monoid~$\MM$, we say that reduction is \emph{$\nn$-semi-convergent} for~$\MM$ if~\eqref{E:NC0} holds for every multifraction~$\aav$ with $\dh\aav \le \nn$.
\end{defi}

We prove in~\cite{Diu} that $\nn$-semi-convergence implies~$(\nn + 1)$-semi-convergence for $\nn = 2, 4$, and conjecturally for every even~$\nn$, so we shall only consider even indices. In this way, we obtain an infinite sequence of stronger and stronger approximations. The following results are established in~\cite{Dit}:

\begin{prop}\label{P:WC}
Let~$\MM$ be a gcd-monoid.

\ITEM1 Reduction is $2$-semi-convergent for~$\MM$ if and only if $\MM$ embeds into~$\EG\MM$.

\ITEM2 Reduction is $4$-semi-convergent for~$\MM$ if and only if $\MM$ embeds into~$\EG\MM$ and every right fraction $\aa\bb\inv$ in~$\EG\MM$ admits a unique expression with $\aa \gcdt \bb = 1$.
\end{prop}

\section{Interval monoids}\label{S:OrdMon}

Our examples and counter-examples involve monoids that are obtained in a uniform way from finite posets, and that are special cases of the monoids investigated in~\cite{Weh}. Here we describe those monoids and, in particular, we recall the characterization, obtained in~\cite{Weh}, of the posets~$\PP$ of which the associated monoid is a gcd-monoid.
We also characterize those~$\PP$ of which the associated monoid is noetherian.

\subsection{Intervals in a poset}\label{SS:Intervals}

By default, the order of all considered posets is denoted by~$\le$, and $<$ is the associated strict ordering. For $(\PP, \le)$ a poset and $\xx \le \yy$ in~$\PP$, we denote by~$[\xx, \yy]$ the \emph{interval} determined by~$\xx$ and~$\yy$, namely $\{\zz \in \PP \mid \xx \le \zz \le \yy\}$. We then put $\xx := \sr{([\xx, \yy])}$ (the \emph{source}) and $\yy := \tg{([\xx, \yy])}$ (the \emph{target}).
We say that the interval $[\xx,\yy]$ is \emph{proper} if $\xx\neq\yy$.

\begin{defi}\label{E:IntMon}
The \emph{interval monoid}~$\IM\PP$ of a poset~$\PP$ is the monoid defined by generators $[\xx,\yy]$, where $\xx,\yy\in P$ with $\xx\le\yy$, and relations
\begin{equation}\label{Eq:IntMon}
[\xx,\zz]=[\xx,\yy]\cdot[\yy,\zz]\,,\text{ for $\xx\le\yy\le\zz$ in~$P$}, \quad
[\xx,\xx]=1\,,\text{ for $\xx$ in~$P$.}
 \end{equation}
\end{defi}

The following statement gathers some elementary properties that are valid in every interval monoid, in particular the existence of a distinguished decomposition in terms of the generators of~\eqref{Eq:IntMon}. It is contained in Lemma~3.4 and Proposition~7.7 of~\cite{Weh}.

\begin{prop}\label{P:NF}
\ITEM1 For every poset~$\PP$, the monoid~$\IM\PP$ embeds into its enveloping group. It is cancellative, and~$1$ is the only invertible element in~$\IM\PP$.

\ITEM2 Call a sequence $(\II_1 \wdots \II_\pp)$ of proper intervals \emph{normal} if $\tg{\II_\ii} \not= \sr{\II_{\ii + 1}}$ holds for every $\ii < \pp$. Then every nontrivial element of~$\IM\PP$ admits a unique expression as $\II_1 \pdots \II_\pp$ with $(\II_1 \wdots \II_\pp)$ normal.
\end{prop}

\begin{proof}
We begin with the existence of a normal decomposition. Let $\aa$ be a nontrivial element of~$\IM\PP$. By definition, $\aa$ can be decomposed into a nonempty product of proper intervals. Starting from such a decomposition and iteratively replacing any length two subsequence of the form $([\xx, \yy], [\yy, \zz])$ with the corresponding length one sequence $([\xx, \zz])$, one necessarily obtains after finitely many steps a normal sequence which, by construction, is again a decomposition of~$\aa$.

Next, denote by~$\FG{\PP}$ the free group based on~$\PP$. For $\xx < \yy \in \PP$, put 
\begin{equation}\label{E:NF}
\phi([\xx, \yy]) = \xx\inv \yy.
\end{equation}
Extend~$\phi$ to a morphism~$\phi^*$ from the free monoid on the intervals of~$\PP$ to~$\FG{\PP}$. Then $\phi^*$ is invariant under the relations of~\eqref{Eq:IntMon}, hence it induces a well defined morphism, still denoted by~$\phi$, from~$\IM\PP$ to~$\FG{\PP}$. 

Now assume $\aa = [\xx_1, \yy_1] \pdots [\xx_\pp, \yy_\pp]$ with $([\xx_1, \yy_1] \wdots [\xx_\pp, \yy_\pp])$ normal. Then we find $\phi(\aa) = \xx_1\inv \yy_1 \pdots \xx_\pp\inv \yy_\pp$, a freely reduced word in~$\FG{\PP}$. This first shows that we can recover the normal decomposition of~$\aa$ from~$\phi(\aa)$, which implies the uniqueness of the latter normal decomposition. Next, this proves that the morphism~$\phi$ is injective on~$\IM\PP$, which implies that the monoid~$\IM\PP$ embeds into a (free) group. By the universal property of the enveloping group, this in turn implies that $\IM\PP$ embeds into its enveloping group. From there, it must be cancellative.

Finally, a finite product of proper intervals $[\xx_1, \yy_1] \wdots [\xx_\pp, \yy_\pp]$ with $\pp \ge 1$ may never be~$1$, since this would require $\xx_1 < \yy_1 = \xx_2 < \yy_2 = \pdots = \xx_\pp < \yy_\pp = \xx_1$, contradicting the assumption that $\PP$ is a poset.
\end{proof}

\begin{rema}
 Although the monoid~$\IM\PP$ embeds into a free group, its enveloping group may not be free, even for finite~$P$.
On the other hand, if~$P$ is finite, then the monoid~$\IM\PP$ always embeds into a free monoid; see~\cite{Weh} for more details.
\end{rema}

Hereafter, we denote by~$\NF(\aa)$ the normal decomposition of an element~$\aa$ of~$\IM\PP$, and call its length the \emph{degree} of~$\aa$, denoted by~$\deg(\aa)$.

\subsection{Divisibility in~$\IM\PP$}

Via the normal decomposition, the divisibility relations of an interval monoid reduce to the prefix and suffix ordering of intervals, respectively.

\begin{lemm}\label{L:Div}
\ITEM1 If $\PP$ is a poset and $\II, \JJ$ are proper intervals of~$\PP$, then $\II$ left divides~$\JJ$ \pup{\resp right divides} in~$\IM\PP$ if and only if we have $\sr\II = \sr\JJ$ and $\tg\II \le \tg\JJ$ \pup{\resp $\sr\II \ge \sr\JJ$ and $\tg\II = \tg\JJ$} in~$\PP$.

\ITEM2 If $\aa, \bb$ belong to~$\IM\PP$, with $\NF(\aa) = (\II_1 \wdots \II_\pp)$ and $\NF(\bb) = (\JJ_1 \wdots \JJ_\qq)$, then $\aa$ left divides \pup{\resp right divides}~$\bb$ in~$\IM\PP$ if and only if we have $\pp \le \qq$, $\II_\kk = \JJ_\kk$ \pup{\resp $\II_{\pp - \kk + 1} = \JJ_{\qq - \kk + 1}$} for $1 \le \kk < \pp$, and $\II_\pp \le \JJ_\pp$ \pup{\resp $\II_1 \divet \JJ_{\qq - \pp + 1}$}. 
\end{lemm}

\begin{proof}
The verification of~\ITEM1 is straightforward. For~\ITEM2, we observe that $\NF(\aa \opp \cc)$ either is the concatenation of the two sequences~$\NF(\aa)$ and~$\NF(\cc)$, or it is obtained from this concatenation by merging the last interval of~$\NF(\aa)$ with the first interval of~$\NF(\cc)$, when the latter match. Expanding $\bb = \aa\cc$ gives the result for~$\dive$.
\end{proof}

An interval monoid~$\IM\PP$ need not always be a gcd-monoid, but we show now that some simple conditions on the poset~$\PP$ are sufficient. The result below can be established as a straightforward application of~\cite[Thm.\,5.8]{Weh}; for convenience sake, we give here a simple direct verification.

For $\xx$ in a poset~$\PP$, we put $\SUP\PP\xx:= \{\yy \in \PP \mid \yy \ge \xx\}$ and $\INF\PP\xx:= \{\yy \in \PP \mid \yy \le \xx\}$.

\begin{defi}
A poset~$\PP$ is said to be a \emph{local lattice} if, for every~$\xx$ in~$\PP$, the induced poset~$\SUP\PP\xx$ is a meet-semilattice, and the induced poset~$\INF\PP\xx$ is a join-semilattice.
\end{defi}

We recall that a poset is a \emph{meet-semilattice} (\resp a \emph{join-semilattice}) if any two elements admit a greatest lower bound (\resp a least upper bound). The following result is contained in \cite[Prop.\,7.9]{Weh}.

\begin{prop}\label{P:Gcd}
For every poset~$\PP$, the monoid~$\IM\PP$ is a gcd-monoid if and only if $\PP$ is a local lattice.
\end{prop}

\begin{proof}
Assume that $\PP$ is local lattice, and $\aa, \bb$ are distinct elements of~$\IM\PP$. Let $(\II_1 \wdots \II_\pp)$ and $(\JJ_1 \wdots \JJ_\qq)$ be the normal forms of~$\aa$ and~$\bb$, respectively. Assume $\II_\kk = \JJ_\kk$ for $\kk < \rr$ and $\II_\rr \not= \JJ_\rr$ (such an $\rr$ exists, since $\aa$ and $\bb$ are distinct). For $\sr{\II_\rr} \not= \sr{\JJ_\rr}$, Lemma~\ref{L:Div}\ITEM2 directly implies that $\II_1 \pdots \II_{\rr - 1}$ is a left gcd of~$\aa$ and~$\bb$ in~$\IM\PP$. Otherwise, let $\xx = \sr{\II_\rr} = \sr{\JJ_\rr}$, and $\yy = \tg{\II_\rr}$, $\zz = \tg{\JJ_\rr}$. Then $\yy$ and~$\zz$ lie in~$\SUP\PP\xx$, hence they admit a greatest common lower bound, say~$\tt$. We claim that $\cc = \II_1 \pdots \II_{\rr - 1} \opp [\xx, \tt]$ is a left gcd of~$\aa$ and~$\bb$ in~$\IM\PP$. Indeed, $\tt \le \yy$ in~$\PP$ implies $[\xx, \tt] \dive [\xx, \yy] = \II_\rr$ in~$\IM\PP$, whence
$$\cc = \II_1 \pdots \II_{\rr - 1} \opp [\xx, \tt] \dive \II_1 \pdots \II_\rr \dive \II_1 \pdots \II_\pp = \aa,$$
and, similarly, $\cc \dive \bb$. On the other hand, assume $\dd \dive \aa$ and $\dd \dive \bb$. Let $(\KK_1 \wdots \KK_\ss)$ be the normal form of~$\dd$. By Lemma~\ref{L:Div}\ITEM2, the assumption $\dd \dive \aa$ implies $\ss \le \pp$ and $\KK_1 = \II_1$, \dots, $\KK_{\ss-1} = \II_{\ss-1}$, $\KK_\ss \dive \II_\ss$. In the case $\ss < \rr$, we directly deduce $\dd \dive \cc$. Assume now $\ss = \rr$. Then we have $\KK_\rr \dive \II_\rr = [\xx, \yy]$, which implies $\KK_\rr = [\xx, \uu]$ for some~$\uu \le \yy$. Arguing similarly from the assumption $\dd \dive \bb$, we obtain $\uu \le \zz$. As $\tt$ is the greatest lower bound of~$\yy$ and~$\zz$ in~$\PP$, we deduce $\uu \le \tt$, whence $\KK_\rr \dive [\xx, \tt]$ and, from there, $\dd \dive \cc$ in~$\IM\PP$. Finally, $\ss > \rr$ is impossible, since it would require $\KK_\rr = \II_\rr$ and $\KK_\rr = \JJ_\rr$, whereas $\II_\rr \not= \JJ_\rr$ holds. Hence $\cc$ is a left gcd of~$\aa$ and~$\bb$ in~$\IM\PP$. The argument for right gcds is symmetric, using the assumption that $\INF\PP\xx$ is a join-semilattice. So $\PP$ being a local lattice implies that $\IM\PP$ is a gcd-monoid.

Conversely, assume that $\IM\PP$ is a gcd-monoid, $\xx$ lies in~$\PP$, and $\yy, \zz$ belong to~$\SUP\PP\xx$. Then the elements~$[\xx, \yy]$ and~$[\xx, \zz]$ of~$\IM\PP$ admit a left gcd. As the latter left divides the interval~$[\xx, \yy]$, Lemma~\ref{L:Div} implies that it is an interval, say~$[\xx, \tt]$. We claim that $\tt$ is a greatest lower bound for~$\yy$ and~$\zz$ in~$\SUP\PP\xx$. First $[\xx, \tt] \dive [\xx, \yy]$ in~$\IM\PP$ implies $\tt \le \yy$ in~$\PP$ and, similarly, $\tt \le \zz$. On the other hand, assume $\uu \le \yy$ and $\uu \le \zz$ in~$\PP$. Then, in~$\IM\PP$, we have $[\xx, \uu] \dive [\xx, \yy]$ and $[\xx, \uu] \dive [\xx, \zz]$, whence $[\xx, \uu] \dive [\xx, \tt]$, and, from there, $\uu \le \tt$ in~$\PP$. So $\SUP\PP\xx$ is a meet-semilattice. Arguing symmetrically from right gcds in~$\IM\PP$, we obtain that $\INF\PP\xx$ is a join-semilattice. So $\IM\PP$ being a gcd-monoid implies that $\PP$ is a local lattice. 
\end{proof}

\subsection{Noetherianity}

We turn to the possible noetherianity of the monoid~$\IM\PP$. We write $\xx \prec \yy$ when $\yy$ is an \emph{immediate successor} of~$\xx$, that is, $x<y$ holds and no element~$\zz$ satisfies $\xx < \zz < \yy$. 

\begin{prop}\label{P:Noeth}
For every poset~$\PP$, the monoid~$\IM\PP$ is noetherian if and only if for every~$\xx$ in~$\PP$, there is no infinite descending chain in~$\SUP\PP\xx$, and no infinite ascending chain in~$\INF\PP\xx$. Its atoms are then the intervals~$[\xx, \yy]$ with $\xx \prec \yy$ \pup{\emph{elementary} intervals}.
\end{prop}

\begin{proof}
By Lemma~\ref{L:Div}\ITEM1, if $[\xx, \yy]$ and $[\xx, \yy']$ are intervals with the same source, then $[\xx, \yy] < [\xx, \yy']$ holds in~$\IM\PP$ if and only if $\yy < \yy'$ holds in~$\PP$. Hence the non-existence of an infinite descending chain with respect to proper left divisibility inside the family of intervals starting at~$\xx$ is equivalent to the non-existence of an infinite descending chain in~$\SUP\PP\xx$. 

By Lemma~\ref{L:Div}\ITEM2, $\aa \le \bb$ implies that the sequence $\NF(\aa)$ is lexicographically smaller than~$\NF(\bb)$, meaning that it is either a prefix or that there exists~$\ii$ such that the first $\ii - 1$ entries coincide and the $\ii$th entry for~$\aa$ left divides the $\ii$st entry for~$\bb$. By the remark above, the left divisibility order on intervals is well-founded. By standard arguments, this implies that its lexicographical extension is well-founded as well, implying that left divisibility has no infinite descending sequence in~$\IM\PP$. 

The argument for right divisibility is symmetric, a descending sequence in intervals with given target being discarded because it would entail an infinite ascending sequence of the sources of the intervals. Hence the monoid~$\IM\PP$ is noetherian (but not necessarily strongly noetherian when $\PP$ is infinite).

The characterization of atoms follows from the definitions directly.
\end{proof}

\begin{coro}\label{C:FiniteNoeth}
For every finite poset~$\PP$, the monoid~$\IM\PP$ is noetherian.
\end{coro}

\begin{rema}\label{R:Garside}
The intervals in a monoid~$\IM\PP$ form a Garside family, and what is called normal decomposition above is the normal form associated with that Garside family in the sense of~\cite[Ch.\,III]{Dir}. Proposition~\ref{P:Noeth} is then an instance of the general result that a monoid with a locally noetherian Garside family is noetherian. A specificity of the monoids~$\IM\PP$ is that the family of intervals is a bilateral Garside family, meaning a Garside family both with respect to left greedy and to right greedy decompositions.
\end{rema}

\section{Embedding into the enveloping group}\label{S:Embed}

We turn to the specific investigation of multifraction reduction in the case of interval monoids. We begin with $2$-semi-convergence, that is, with the embeddability of the monoid into its group. Our aim is to prove Proposition~B, that is, to construct an example of a gcd-monoid that does not embed into its group.

\subsection{Malcev conditions}\label{SS:Malcev}

We aim at constructing a gcd-monoid~$\MM$ that does not embed into its group. Note that, by Proposition~\ref{P:NF}, $\MM$ cannot be the interval monoid of a poset.

It is known that a monoid embeds into its group if and only if it is cancellative and satisfies an infinite list of quasi-identities known as \emph{Malcev conditions}~\cite{Malcev39}, see~\cite[Ch.\,12]{ClP2} and~\cite{Bouleau, Oso}. Malcev conditions are encoded in \emph{Malcev words}, which are those words in the letters~$L_i, L_i^*, R_i, R^*_i$, $\ii \ge 1$ that obey some syntactic constraints described in~\cite[p.\,310]{ClP2}. Then one can show that a gcd-monoid satisfies a number of Malcev conditions.

\begin{prop}\label{P:Malcev}
Assume that $\MM$ is a gcd-monoid satisfying all Malcev conditions encoded in Malcev words of length at most~$2\ell$. Then $\MM$ satisfies all Malcev conditions encoded in a Malcev word of length~$\ell + 2$ that contains a factor of the form~$L_i R_j L_i^*$ or $R_i L_j R_i^*$.
\end{prop}

\begin{proof}
We only treat one simple instance, the general scheme being similar. Consider the Malcev word $L_1R_1L_1^*R_1^*$. With the notation of~\cite[p.\,310]{ClP2}, the corresponding quasi-identity is
\begin{equation}\label{E:Malcev}
( \dd \aa = \AA \CC \quand \dd \bb = \AA \DD \quand \cc \bb = \BB \DD ) \Rightarrow \cc \aa = \BB \CC.
\end{equation}
Assume that $\aa \wdots \DD$ satisfy the three left equations of~\eqref{E:Malcev}, see Figure~\ref{E:Malcev}. Let $\ee:= \AA \gcd \dd$, and define~$\AA'$ and~$\dd'$ by $\AA = \ee \AA'$, $\dd = \ee \dd'$. Then $\dd\aa = \AA\CC$ expands into $\ee\dd'\aa = \ee\AA' \CC$, whence $\dd'\aa = \AA'\CC$ by left cancelling~$\ee$. Similarly, $\dd\bb = \AA\DD$ implies $\dd'\bb = \AA'\DD$. As $\ee$ is the left gcd of~$\AA$ and~$\dd$, we must have $\dd' \gcd \AA' = 1$. By standard arguments (see for instance~\cite[Lemma~2.12]{Dit}), this with $\dd' \bb = \AA' \DD$ implies that $\dd'\bb$ is the left lcm of~$\bb$ and~$\DD$. As we have $\cc\bb = \BB\DD$, this implies that $\dd' \bb$ right divides~$\cc\bb$. Hence we have $\cc \bb = \ff \dd' \bb$ for some~$\ff$, which implies $\cc = \ff \dd'$ by right cancelling~$\bb$. Similarly, $\BB\DD = \ff \AA' \DD$ implies $\BB = \ff \AA'$. But then we deduce $\cc\aa = \ff\dd' \aa = \ff \AA' \CC = \BB\CC$, which proves~\eqref{E:Malcev}.
\end{proof}

\begin{rema}\label{R:Interpol}
The above argument remains valid when the assumption that $\MM$ admits gcds is relaxed into the condition that $\MM$ satisfies the interpolation property: if $\aa$ and~$\bb$ are common right multiple of~$\cc$ and~$\dd$, then there exists a common multiple~$\ee$ of~$\cc$ and~$\dd$ of which $\aa$ and~$\bb$ are multiples.
\end{rema}

\begin{figure}[htb]
\begin{picture}(60,26)(0,0)
\psset{nodesep=0.5mm}
\pcline[border=1mm,style=exist]{->}(0,24)(35,12)\taput{$\ee$}
\pcline{->}(0,24)(40,0)\put(3,17){$\AA$}
\pcline[border=1mm,style=exist]{->}(0,0)(35,12)\tbput{$\ff$}
\pcline{->}(40,0)(60,24)\put(57,17){$\DD$}
\pcline{->}(0,24)(40,24)\taput{$\dd$}
\pcline{->}(40,24)(60,24)\taput{$\bb$}
\pcline{->}(40,0)(60,0)\tbput{$\CC$}
\psline(0,0)(15,9)\pcline[border=1mm]{->}(10,6)(40,24)\put(4,5){$\cc$}
\pcline{->}(0,0)(40,0)\tbput{$\BB$}
\pcline[border=1mm]{->}(40,24)(60,0)\put(57,5){$\aa$}
\pcline[style=exist]{->}(35,12)(40,24)\put(38,16){$\dd'$}
\pcline[style=exist]{->}(35,12)(40,0)\put(37.5,7){$\AA'$}
\psarc[style=thinexist](35,12){3}{-70}{70}
\end{picture}
\caption{\sf Proof of the first Malcev condition in a gcd-monoid.}
\label{F:Malcev}
\end{figure}
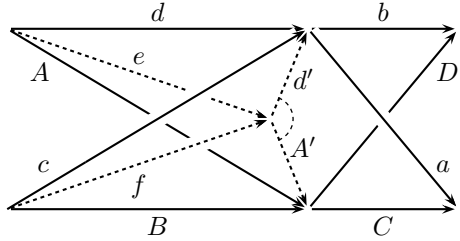

\subsection{A counter-example}

We now establish Proposition~B of the introduction. Proposition~\ref{P:Malcev} implies that, if a gcd-monoid fails to satisfy some Malcev condition, the latter has to be complicated. The monoid we construct below turns out to miss the Malcev condition encoded in $L_1 R_1 R_2 L_1^* R_2^* L_3 R_1^* R_3 L_3^* L_2^* R_3^*$, involving 24~variables and $11 + 1$ equalities (and not eligible for Proposition~\ref{P:Malcev}).

\begin{prop}\label{P:PropB}
\pup{See Figure~\ref{F:PropB}.}
Let $\Omega:= \{1, 2, 3, 4\}$, let $\PPB$ be the $14$-element poset $(\Pw(\Omega) \setminus \{\emptyset, \Omega\}, \subseteq)$, and let $\MMB$ admit the presentation obtained from the presentation~\eqref{Eq:IntMon} of~$\IM\PPB$ by deleting $[1, 12][12, 123] = [1, 13][13,123]$. Then~$\MMB$ is a noetherian gcd-monoid failing to embed into its group.
\end{prop}

\begin{proof}
The poset $(\Pw(\Omega), \subseteq)$ is a lattice, hence every subset $\SUP{\Pw(\Omega)}\xx$ is a meet-semilattice, and therefore so is every subset $\SUP\PPB\xx$, since the latter is an initial subset of~$\SUP{\Pw(\Omega)}\xx$. Similarly, every subset $\SUP\PPB\xx$ is a join-semilattice, and $\PPB$ is a local lattice. Hence, by Proposition~\ref{P:Gcd}, $\IM\PPB$ is a gcd-monoid. This however says nothing \emph{a priori} about~$\MMB$, of which $\IM\PPB$ is a quotient.
A possibility is then to analyze the monoid~$\MMB$ \emph{via} the results of \cite[Sec.\,II.4]{Dir}.
Another possibility, perhaps requiring less calculations, is the following.

Set $u=1$ and $v=123$.
Following the terminology of~\cite[Sec.\,8]{Weh}, the closed interval $[1,123]$ of~$\PPB$ is an \emph{extreme spindle}, that is, $u$ is minimal, $v$ is maximal, there exists~$z$ such that $u<z<v$, and the comparability relation on the open interval $(u,v)$ (here, reduced to the two sets~$z_2=12$ and~$z_3=13$) is an equivalence relation.
By \cite[Prop.\,9.6]{Weh}, $\MMB$ is isomorphic to the monoid denoted there by $\IM{\PPB,u,v}$.
By \cite[Prop.\,9.5]{Weh}, it follows that~$\MMB$ is a gcd-monoid.
Since~$\MMB$ admits a presentation by homogeneous relations, it is noetherian.
Furthermore, in the monoid $\IM{\PPB,u,v}$, the products $[u,z_i][z_i,v]$, for $i\in\{2,3\}$, are respectively equal to the maximal chains $\{u,z_2,v\}$ and $\{u,z_3,v\}$, thus they are distinct.
Hence the relation $[1, 12][12, 123] = [1, 13][13,123]$ fails in~$\MMB$.

On the other hand, one easily checks on the right hand side diagram of Figure~\ref{P:PropB} that the relation holds in the group~$\EG{\MMB}$, as the following derivation shows:
\begin{align*}
[1, 12][12, 123]
&= [1, 12] [12, 123] [23, 123]\inv [23, 123]\\
&= [1, 12][2,12]\inv [2, 23] [23, 123]\\
&= [1, 12][2,12]\inv [2, 23] [23, 234] [23, 234]\inv [23, 123]\\
&= [1, 12][2,12]\inv [2, 24] [24, 234] [23, 234]\inv [23, 123]\\
&= [1, 12] [12, 124] [24,124]\inv [24, 234] [23, 234]\inv [23, 123]\\
&= [1, 14] [14, 124] [24,124]\inv [24, 234] [23, 234]\inv [23, 123]\\
&= [1, 14] [4, 14]\inv [4, 24] [24, 234] [23, 234]\inv [23, 123]\\
&= [1, 14] [4, 14]\inv [4, 34] [34, 234] [23, 234]\inv [23, 123],
\end{align*}
from which one returns to $[1, 13][13, 123]$ by a symmetric derivation where $2$ and~$3$ are interchanged, according to the symmetry of the diagram. Hence $\MMB$ does not embed into its enveloping group and, therefore, it is an example of a (noetherian) gcd-monoid for which reduction is not $2$-semi-convergent.
\end{proof}

\begin{figure}[htb]
\def\AArrow(#1,#2){\ncline[nodesep=0.3mm, linewidth=0.8pt, border=0.8pt]{->}{#1}{#2}}
\def\CArrow(#1,#2){\ncline[linecolor=red, nodesep=0.3mm, linewidth=0.8pt, border=0.8pt]{->}{#1}{#2}}
\def\PArrow(#1,#2){\ncline[linestyle=dashed, nodesep=0.3mm, linewidth=0.8pt, border=0.8pt]{#1}{#2}}
\begin{picture}(62,65)(0,-12)
\psset{nodesep=0.5mm}
\psset{unit=1.2mm}
\WPoint(25,-10,0)\nput{210}{0}{$\scriptstyle\emptyset$}
\WPoint(25,40,1234)\nput{90}{1234}{$\scriptstyle1234$}
\WPoint(5,0,1)\nput{270}{1}{$\scriptstyle1$}
\WPoint(20,0,2)\nput{270}{2}{$\scriptstyle2$}
\WPoint(30,0,3)\nput{270}{3}{$\scriptstyle3$}
\WPoint(45,0,4)\nput{270}{4}{$\scriptstyle4$}
\WPoint(0,15,12)\nput{180}{12}{$\scriptstyle12$}
\WPoint(10,15,13)\nput{180}{13}{$\scriptstyle13$}
\WPoint(20,15,14)\nput{180}{14}{$\scriptstyle14$}
\WPoint(30,15,23)\nput{0}{23}{$\scriptstyle23$}
\WPoint(40,15,24)\nput{0}{24}{$\scriptstyle24$}
\WPoint(50,15,34)\nput{0}{34}{$\scriptstyle34$}
\WPoint(5,30,123)\nput{120}{123}{$\scriptstyle123$}
\WPoint(20,30,124)\nput{120}{124}{$\scriptstyle124$}
\WPoint(30,30,134)\nput{60}{134}{$\scriptstyle134$}
\WPoint(45,30,234)\nput{45}{234}{$\scriptstyle234$}
\AArrow(1,14) \AArrow(2,12) \AArrow(2,23) \AArrow(2,24) \AArrow(3,13) \AArrow(3,23) \AArrow(3,34) \AArrow(4,14) \AArrow(4,24) \AArrow(4,34) \CArrow(1,12)\CArrow(1,13) 
\CArrow(12,123) \AArrow(12,124)
\CArrow(13,123) \AArrow(13,134)
\AArrow(14,124) \AArrow(14,134)
\AArrow(23,123) \AArrow(23,234)
\AArrow(24,124) \AArrow(24,234)
\AArrow(34,134) \AArrow(34,234)
\PArrow(0,1) \PArrow(0,2) \PArrow(0,3) \PArrow(0,4) 
\PArrow(123,1234) \PArrow(124,1234) \PArrow(134,1234) \PArrow(234,1234) 
\end{picture}
\begin{picture}(60,65)(-2,-3)
\psset{nodesep=0.5mm}
\psset{unit=1.10mm}
\WPoint(20,20,1)\nput{270}{1}{$\scriptstyle1$}
\WPoint(15,40,2)\nput{60}{2}{$\scriptstyle2$}
\WPoint(40,15,3)\nput{30}{3}{$\scriptstyle3$}
\WPoint(2,2,4)\nput{240}{4}{$\scriptstyle4$}
\WPoint(20,30,12)\nput{45}{12}{$\scriptstyle12$}
\WPoint(30,20,13)\nput{45}{13}{$\scriptstyle13$}
\WPoint(10,10,14)\nput{300}{14}{$\scriptstyle14$}
\WPoint(40,40,23)\nput{340}{23}{$\scriptstyle23$}
\WPoint(2,48,24)\nput{110}{24}{$\scriptstyle24$}
\WPoint(48,2,34)\nput{300}{34}{$\scriptstyle34$}
\WPoint(30,30,123)\nput{315}{123}{$\scriptstyle123$}
\WPoint(10,35,124)\nput{210}{124}{$\scriptstyle124$}
\WPoint(35,10,134)\nput{240}{134}{$\scriptstyle134$}
\WPoint(48,48,234)\nput{60}{234}{$\scriptstyle234$}
\CArrow(1,12)\CArrow(1,13) \AArrow(1,14) \AArrow(2,12) \AArrow(2,23) \AArrow(2,24) \AArrow(3,13) \AArrow(3,23) \AArrow(3,34) \AArrow(4,14) \AArrow(4,24) \AArrow(4,34) 
\CArrow(12,123) \AArrow(12,124)
\CArrow(13,123) \AArrow(13,134)
\AArrow(14,124) \AArrow(14,134)
\AArrow(23,123) \AArrow(23,234)
\AArrow(24,124) \AArrow(24,234)
\AArrow(34,134) \AArrow(34,234)
\end{picture}
\caption{\sf The poset~$\PPB$, viewed as a truncated $4$-cube (left) and as a planar graph (right); one easily sees, especially on the right hand side diagram, that the colored relation can be deduced from the other eleven relations in any group, that is, when one can cross the arrows of the diagram in the opposite direction.}
\label{F:PropB}
\end{figure}
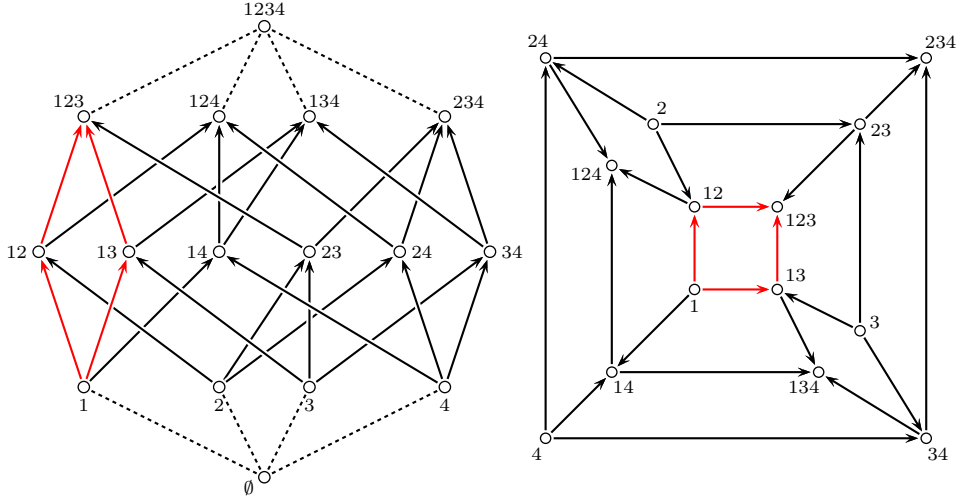

As reduction is not $2$-semi-convergent for~$\MMB$, it cannot be semi-convergent either: in the current case, the $2$-multifraction $[1, 12][12,123] / [1, 13][13, 123]$ is unital, and it is irreducible. It is easy to deduce counter-examples to other properties considered in~\cite{Dit}. For instance, the $6$-multifraction 
$$[1, 12][12, 123]/ [23, 123] / [23, 234] / [4, 234] / [4, 14] / [1, 14]$$
reduces both to~$\one$ and to~$[1, 12][12, 123] / [1, 134] \opp \one$, contradicting what is called $\one$-confluence in~\cite{Dit}.

The monoid~$\MMB$ has $24$~atoms. By a systematic computer search, one can find quotients of~$\MMB$ with similar properties, for instance
$$\begin{matrix}
\langle\tta \wdots \mathtt{k} \mid \mathtt{ab=ba, bc=cb, cd=dc, de=ed, eb=ih,}\hspace{3cm}\\ 
\hspace{3cm}\mathtt{fc=cg, fa=dh, hd=ij, hg=kb, dj=ic, ie=kf}\rangle^+,
\end{matrix}$$
with $11$ atoms, the missing relation~$\mathtt{ad=cf}$, and the $6$-multifraction $\mathtt{ad/e/j/cd/f/b}$ that reduces both to~$\one$ and to~$\mathtt{ad/cf/1/1/1/1}$.

\section{Semi-convergence in interval monoids}\label{S:NCInt}

We now establish sufficient conditions implying that reduction is semi-convergent for the interval monoid of a poset~$\PP$, and derive explicit examples where these conditions are satisfied.

\subsection{Minimal multifractions}\label{SS:Ineludible}

\emph{A priori}, semi-convergence is an infinitary property, and we first introduce an induction scheme that may, in good cases, reduce it to finitary conditions. To this end, we first introduce a partial ordering on multifractions.

\begin{defi}
Let $\MM$ be a monoid. For~$\aav, \bbv$ in~$\FRb\MM$, say that $\aav$ is a \emph{proper piece} of~$\bbv$, written $\aav \piece \bbv$, if we have $\bbv = \ccv \opp \aav \opp \ddv$ for some~$\ccv, \ddv$ that are not both trivial. We say that $\aav$ is a \emph{piece} of~$\bbv$ if we have either $\aav \piece \bbv$ or $\aav = \bbv$.
\end{defi}

Thus $\piece$ is the factor relation of the monoid~$\FRb\MM$. Note that $\piece$ is not the lexicographical extension of the factor relation of the monoid~$\MM$.

\begin{lemm}\label{L:PieceWF}
If $\MM$ is a noetherian monoid, then the relation~$\piece$ is a well-founded strict partial ordering.
\end{lemm}

\begin{proof}
That $\piece$ is transitive follows from the associativity of the product. Next, assume $\aav = \ccv \opp \aav \opp \ddv$. If $\ccv$ is not~$\ef$, the only possibility for guaranteeing $\dh\aav = \dh{\ccv \opp \aav \opp \ddv}$ is that $\ccv$ has depth one, with the sign of~$\aa_1$. In that case,
$\aav = \ccv \opp \aav \opp \ddv$ requires $\aa_1 = \xx_1 \aa_1$ (if $1$ is positive in~$\aav$) or $\aa_1 = \aa_1 \xx_1$ (if $1$ is negative in~$\aav$), hence $\xx_1 = 1$ in both cases, so $\ccv$ must be trivial. The argument is similar for~$\ddv$. Hence $\piece$ is irreflexive, and it is a strict partial ordering on~$\FRb\MM$.

Assume that $\aav^1 \piecet \aav^2 \piecet \pdots$ is an infinite descending chain in~$\FRb\MM$, say $\aav^\ii = \ccv^\ii \opp \aav^{\ii + 1} \opp \ddv^\ii$. By definition, $\aav \piece \bbv$ implies $\dh\aav \le \dh\bbv$, hence the sequence $\dh{\aav^\ii}$ is non-increasing, and there exists~$\NN$ and~$\nn$ such that $\dh{\aav^\ii} = \nn$ holds for every~$\ii \ge \NN$. Now, as above, $\dh\aav = \dh{\ccv \opp \aav \opp \ddv}$ requires $\dh\ccv \le 1$ and $\dh\ddv \le 1$. Assume for instance that $1$ is positive and $\nn$ is negative in~$\aav^\ii$ for~$\ii \ge \NN$. Then $\aav^\ii = \ccv^\ii \opp \aav^{\ii + 1} \opp \ddv^\ii$ implies $\aa^\ii_1 = \cc^\ii_1 \aa^{\ii + 1}_1$, whence $\aa^\ii_1 \multet \aa^{\ii + 1}_1$, and $\aa^\ii_1 \multt \aa^{\ii + 1}_1$ whenever $\ccv^\ii$ is not trivial, and, similarly, $\aa^\ii_\nn =\aa^{\ii + 1}_\nn \dd^\ii_1$, whence $\aa^\ii_\nn \multe \aa^{\ii + 1}_\nn$, and $\aa^\ii_\nn \mult \aa^{\ii + 1}_\nn$ whenever $\ddv^\ii$ is not trivial. So we have $(\aa_1^\ii, \aa_\nn^\ii) \PROD\multt\mult (\aa_1^{\ii + 1}, \aa_\nn^{\ii + 1})$, where, if $O, O'$ are partial orders on a set~$\XX$ and $(\xx, \xx')$, $(\yy, \yy')$ belong to~$\XX^2$, we write $(\xx, \xx') \PROD{O}{O'} (\yy, \yy')$ for the conjunction of $\xx \,O\, \yy$, $\xx' \,O'\, \yy'$, and at least one of $\xx \not= \yy$, $\xx' \not= \yy'$. Standard arguments show that $\PROD{O}{O'}$ is well-founded (no infinite descending sequence) whenever $O$ and~$O'$ are. If $\MM$ is a noetherian monoid, then, by definition, the partial orders~$\mult$ and~$\multt$ are well-founded, and, therefore, so is $\PROD\multt\mult$, hence a descending chain as above cannot exist. The other sign possibilities for~$1$ and~$\nn$ in~$\aav^\ii$ are treated similarly, appealing to the various $\vert$-combinations of~$\mult$ and~$\multt$.\end{proof}

Thus we may develop inductive arguments based on the relation~$\piece$ and, in particular, appeal to $\piece$-minimal multifractions, which must exist.

\begin{prop}\label{P:Base}
Assume that $\MM$ is a noetherian gcd-monoid. Then reduction is semi-convergent for~$\MM$ if and only if
\begin{equation}\label{E:Base}
\parbox{110mm}{Every nontrivial $\piece$-minimal unital multifraction on~$\MM$ is reducible.}
\end{equation}
\end{prop}

\begin{proof}
If reduction is semi-convergent for~$\MM$, then every nontrivial unital multifraction must be reducible, so the condition is necessary. Conversely, assume \eqref{E:Base}. The family of all nontrivial unital pieces of~$\aav$ is nonempty (it contains~$\aav$ at least), hence, by Lemma~\ref{L:PieceWF}, it contains at least one $\piece$-minimal element, say~$\bbv$. Then~\eqref{E:Base} implies that $\bbv$ is reducible. By Lemma~\ref{L:Basic}, this implies that $\aav$ itself is reducible. Hence reduction is semi-convergent for~$\MM$. 
\end{proof}

\begin{coro}\label{C:Decid}
Assume that $\MM$ is a noetherian gcd-monoid that admits finitely many $\piece$-minimal unital multifractions. Then for reduction to be semi-convergent for~$\MM$ is a decidable property.
\end{coro}

\begin{proof}
By Proposition~\ref{P:Base}, deciding whether reduction is semi-convergent for~$\MM$ amounts to checking the reducibility of the finitely many unital multifractions. Now, starting from a finite presentation of~$\MM$, testing the reducibility of one multifraction is a decidable property, see~\cite[Prop.\,3.27]{Dit}.
\end{proof}

If we consider multifractions with bounded depth, since $\aav \piece \bbv$ implies $\dh\aav \le\nobreak \dh\bbv$, we obtain the following local version of Proposition~\ref{P:Base}:

\begin{prop}\label{P:BaseNC}
Assume that $\MM$ is a noetherian gcd-monoid. Then reduction is $\nn$-semi-convergent for~$\MM$ if and only if \eqref{E:Base} restricted to $\piece$-minimal unital multifractions of depth~$\le \nn$ holds.
\end{prop}

\subsection{Simple multifractions}\label{SS:Basic}

We now consider the case of an interval monoid, with the aim of pinpointing a small set of multifractions containing all $\piece$-minimal unital multifractions.

We recall that, if $\PP$ is a poset and $\II$ is an interval of~$\PP$, we write $\sr\II$ (\resp $\tg\II$) for the source (\resp target) of~$\II$. We first extend the notation to nontrivial elements of~$\IM\PP$ by putting $\sr(\aa):= \sr(\II_1)$ and $\tg(\aa):= \tg(\II_\pp)$ for $\aa = \II_1 \pdots \II_\pp$. Then we extend it to multifractions by putting $\sr\aav = \sr\aa_1$ (\resp $\tg\aa_1$) for $\aa_1$ positive (\resp negative), and $\tg\aav = \tg\aa_\nn$ (\resp $\sr\aa_\nn$) for $\aa_\nn$ positive (\resp negative).

\begin{defi}
Let $\PP$ be a local lattice, and $\MM = \IM\PP$. Call a multifraction~$\aav$ on~$\MM$ \emph{simple} if each entry of~$\aav$ is a proper interval and, moreover, $\tg{\aa_\ii} = \tg{\aa_{\ii + 1}}$ (\resp $\sr{\aa_\ii} = \sr{\aa_{\ii + 1}}$) holds for each~$\ii < \dh\aav$ that is positive (\resp negative) in~$\aav$. Write~$\Sim\MM$ for the family of all simple multifractions on~$\MM$ that are unital, and $\Simm\MM$ for the family of all $\piece$-minimal elements of~$\Sim\MM$.
\end{defi}

By definition, the proper intervals and their inverses are simple and, therefore, every multifraction is a finite product of simple multifractions. Exactly as in Proposition~\ref{P:NF} for decomposing the elements of~$\MM$, we have

\begin{lemm}\label{L:SignedNormalDec}
Let $\PP$ be a local lattice. Call a sequence of simple multifractions $(\aav^1 \wdots \aav^\pp)$ \emph{normal} if $\tg\aav^\ii \not= \tg{\aav^{\ii + 1}}$ \pup{\resp $\sr\aav^\ii \not= \sr{\aav^{\ii + 1}}$} holds for each~$\ii < \dh\aav$ such that $\aav^\ii$ is positive \pup{\resp negative}. Then every signed multifraction on~$\IM\PP$ admits a normal decomposition into simple multifractions.
\end{lemm}

\begin{proof}
Start with an arbitrary decomposition of the considered multifraction~$\aav$ as a finite product of simple multifractions. As long as there exist two adjacent entries~$\aav^\kk, \aav^{\kk + 1}$ whose product remains simple, shorten the decomposition by replacing $(\aav^\kk, \aav^{\kk + 1})$ with $(\aav^\kk \opp \aav^{\kk + 1})$. After finitely many steps, one obtains a normal decomposition.
\end{proof}

Our aim is to prove:

\begin{prop}\label{P:PosBase}
If $\PP$ is a local lattice and $\MM$ is~$\IM\PP$, then every $\piece$-minimal unital multifraction is simple.
\end{prop}

The proof relies on the following technical result. The construction of~$\psi$ from~$\iota$, introduced in~\cite{Weh}, is called there the \emph{highlighting expansion}.

\begin{lemm}\label{L:Separ}
If $\PP$ is a local lattice and $\MM$ is~$\IM\PP$, there exists a unique morphism~$\psi$ from the monoid~$\FRb\MM$ to the free product $\FG{\PP} * \EG{\IM\PP}$ satisfying
\begin{equation}\label{E:Separ}
\psi(\aav) = \sr\aav\inv \opp \iota(\aav) \opp \tg\aav.
\end{equation}
for every simple multifraction~$\aav$. A signed multifraction~$\aav$ is unital if and only if $\psi(\aav) = 1$ holds.
\end{lemm}

\begin{proof}
For $\xx < \yy$ in~$\PP$, put $\psi([\xx, \yy]) := \xx\inv \opp \iota([\xx, \yy]) \opp \yy$, an element of the group $\FG{\PP} * \EG{\IM\PP}$, and extend~$\psi$ to the free monoid on the proper intervals of~$\PP$. For $\xx < \yy < \zz \in \PP$, we find
$$\psi([\xx, \yy]) \psi([\yy, \zz]) = \xx\inv \iota([\xx, \yy]) \yy \opp \yy\inv \iota([\yy, \zz]) \zz = \xx\inv \iota([\xx, \zz]) \zz = \psi([\xx, \zz]),$$
hence $\psi$ induces a well defined morphism from~$\MM$ to $\FG{\PP} * \EG{\IM\PP}$. Next, for~$\INV\aa$ in~$\INV\MM$, we put $\psi(\INV\aa) := \psi(\aa)\inv$, and we extend~$\psi$ multiplicatively to a morphism from the monoid~$\FRb\MM$ to $\FG{\PP} * \EG{\IM\PP}$: to check that this morphism is well defined, it suffices to consider the four products $\psi(\aa^\eps)\psi(\bb^{\eps'}) = \psi(\aa^\eps \opp \bb^{\eps'})$ for $\aa, \bb$ in~$\MM$ and $\eps, \eps'$ in~$\{\pm\}$ (where $\aa^+$ stands for~$\aa$ and $\aa^-$ for $\INV\aa$), which is straightforward. By construction, \eqref{E:Separ} is valid for the intervals and their inverses, and an obvious induction on the number of intervals extends the equality to all simple multifractions. Thus $\psi$ exists as expected. As simple multifractions generate~$\FRb\MM$, uniqueness is clear.

By Proposition~\ref{P:EnvGroup}, a multifraction~$\aav$ is unital, that is, $\can(\aav) = 1$ holds, if and only if we have $\aav \simeqb 1$, where $\simeqb$ is generated by the pairs $(1, \ef)$, $(\aa / \aa, \ef)$, and $(/ \aa / \aa, \ef)$ with $\aa$ in~$\MM$. We have $\psi(\ef) = 1$ and $\psi$ is a morphism, so, in order to prove that $\can(\aav) = 1$ implies $\psi(\aav) = 1$, it is enough to check the values $\psi(\aa / \aa) = 1$ and $\psi(/ \aa / \aa) = 1$, which is straightforward. 
 
Conversely, $\can(\aav)$ is the projection of~$\psi(\aav)$ obtained by collapsing the elements of~$\FG{\PP}$ when~$\aav$ is simple, whence for every~$\aav$. Hence $\psi(\aav) = 1$ implies~$\can(\aav) = 1$.
\end{proof}

We can now complete the argument.

\begin{proof}[Proof of Proposition~\ref{P:PosBase}] 
Put $\MM:= \IM\PP$. Assume that $\aav$ is a nontrivial unital multifraction on~$\MM$. Let $(\aav^1 \wdots \aav^\pp)$ be a normal decomposition of~$\aav$ as provided by Lemma~\ref{L:SignedNormalDec}. If $\pp = 1$ holds, then $\aav$ is simple, and we are done. So assume $\pp \ge 2$. By~\eqref{E:Separ}, the equality
\begin{equation}\label{E:PosBase}
(\sr{\aav^1})\inv \opp \can(\aav^1) \opp (\tg{\aav^1}) \opp (\sr{\aav^2})\inv \opp \can(\aav^2) \opp (\tg{\aav^2}) \pdots (\sr{\aav^\pp})\inv \opp \can(\aav^\pp) \opp (\tg{\aav^\pp}) = 1
\end{equation}
holds in~$\FG{\PP} * \EG{\MM}$. If none of the elements~$\aav^\ii$ is unital, the word on the left hand side of~\eqref{E:PosBase} is already in normal form in the free product $\FG{\PP} * \EG{\MM}$, and therefore it cannot represent~$1$ in this free product, which contradicts the assumption that $\aav$ is unital. So at least one of the~$\aav^\ii$ is unital. Hence, every nontrivial unital multifraction has a piece that is both simple and unital and, therefore, every nontrivial $\piece$-minimal unital multifraction must be simple.
\end{proof}

Merging Propositions~\ref{P:Base} and~\ref{P:PosBase}, we deduce 

\begin{coro}\label{C:Base}
If $\MM$ is the interval monoid of a finite local lattice~$\PP$, then reduction is semi-convergent for~$\MM$ if and only if 
\begin{equation}\label{E:WCSimple}
\text{Every multifraction in~$\Simm\MM$ is reducible.}
\end{equation}
\end{coro}

\subsection{Loops in the Hasse diagram}\label{SS:Circuits}

If $\PP$ is a local lattice, then the simple multifractions on the monoid~$\IM\PP$ are in one-to-one correspondence with the sequences of positive and negative proper intervals with matching ends. It is then easy to translate the notion of multifraction reducibility to the language of loops in~$\PP$, that is, in homotopical terms. 

\begin{defi}
Let $\PP$ be a poset. For $\nn \ge 1$, a \emph{positive \pup{\resp negative} $\nn$-zigzag} in~$\PP$ is a sequence $\xxv = (\xx_0 \wdots \xx_\nn)$ of elements in~$\PP$ such that $\xx_\ii < \xx_{\ii - 1}$ and $\xx_\ii < \xx_{\ii + 1}$ hold for every $\ii$ odd (\resp $\ii$ even).
The zigzag is called \emph{closed} for $\xx_\nn = \xx_0$, and \emph{simple} if $\xx_1 \wdots \xx_\nn$ are pairwise distinct.
\end{defi}

The following correspondence is then straightforward:

\begin{lemm}\label{L:Dict}
Let $\PP$ be a local lattice. Putting
\begin{equation}\label{E:Dict}
\FF((\xx_0 \wdots \xx_\nn)) := [\xx_0, \xx_1] \sdots [\xx_{\nn-1}, \xx_\nn] \quad (\resp / [\xx_0, \xx_1] \sdots [\xx_{\nn-1}, \xx_\nn])
\end{equation} 
defines a one-to-one correspondence between positive \pup{\resp negative} $\nn$-zigzags in~$\PP$ and simple positive \pup{\resp negative} $\nn$-multifrac\-tions in~$\IM\PP$. If, moreover, $\FF(\xxv)$ is unital, then $\xxv$ is closed.
\end{lemm}

\begin{proof}
Put $\MM := \IM\PP$. The correspondence directly follows from the definitions. 

Next, assume that $\FF(\xxv)$ is unital, that is, $\can(\FF(\xxv)) = 1$ holds. By Lemma~\ref{L:Separ}, we have $\psi(\FF(\xxv)) = 1$. On the other hand, by~\eqref{E:Separ}, which is valid since $\FF(\xxv)$ is simple, we have $\psi(\FF(\xxv)) = \xx_0\inv \can(\FF(\xxv)) \xx_\nn = \xx_0\inv \xx_\nn$, whence $\xx_0 = \xx_\nn$. 
%
\end{proof}

It remains to translate the definition of reducibility in the language of zigzags.

\begin{defi}\label{D:PosRed}
(See Figure~\ref{F:PosRed}.) Let $\PP$ be a poset. For $\ii < \nn$, an $\nn$-zigzag~$\xxv$ on~$\PP$ is called \emph{reducible at~$\ii$} if there exist $\yy$ in~$\PP$ satisfying\\
- for $\ii \ge 2$ with $\xx_\ii < \xx_{\ii + 1}$: $\xx_\ii < \yy \le \xx_{\ii + 1}$ and $\xx_{\ii - 1}, \yy$ have a common upper bound,\\
- for $\ii \ge 2$ with $\xx_\ii > \xx_{\ii + 1}$: $\xx_{\ii + 1} \le \yy < \xx_\ii$ and $\xx_{\ii - 1}, \yy$ admit a common lower bound,\\ 
- for $\ii = 1$ with $\xx_0 < \xx_1$: $\xx_0 \le \yy < \xx_1$ and $\xx_2 \le \yy < \xx_1$,\\ 
- for $\ii = 1$ with $\xx_0 > \xx_1$: $\xx_1 < \yy \le \xx_0$ and $\xx_1 < \yy \le \xx_2$.\\
We say that $\xxv$ is \emph{reducible} if it is reducible at at least one level~$\ii$. 
\end{defi}

\begin{figure}[htb]
\def\NPoint(#1,#2,#3){\cnode[style=thin,fillcolor=white,linecolor=white](#1,#2){0}{#3}}
\def\WPoint(#1,#2,#3){\cnode[style=thin,fillcolor=white,fillstyle=solid](#1,#2){0.5}{#3}}
\def\AArrow(#1,#2){\ncline[nodesep=1mm, linewidth=0.8pt, border=2pt]{->}{#1}{#2}}
\def\PArrow(#1,#2){\ncline[linestyle=dashed,linewidth=0.8pt, border=1.2pt]{->}{#1}{#2}}
\begin{picture}(37,30)(0,-4)
\psset{nodesep=0.7mm}
\psset{unit=1.2mm}
\put(-5,24){case $\ii= 1$}
\put(-3,20){with $\xx_0 < \xx_1$:}
\WPoint(0,0,x0)\nput{270}{x0}{$\xx_0$}
\WPoint(10,5,y)\nput{270}{y}{$\yy$}
\WPoint(10,10,x1)\nput{90}{x1}{$\xx_1$}
\WPoint(20,0,x2)\nput{270}{x2}{$\xx_2$}
\NPoint(25,5,C)
\AArrow(x0,x1)
\AArrow(x2,x1)
\AArrow(x2,C)
\PArrow(x0,y)
\PArrow(x2,y)
\PArrow(y,x1)
\end{picture}
\begin{picture}(42,15)(0,-2)
\psset{nodesep=0.7mm}
\psset{unit=1.2mm}
\put(-5,26){case $\ii \ge 2$}
\put(-3,22){with $\xx_\ii < \xx_{\ii + 1}$:}
\NPoint(0,5,A)
\WPoint(5,10,xi-1)\nput{135}{xi-1}{$\xx_{\ii - 1}$}
\WPoint(10,15,B)
\WPoint(15,0,xi)\nput{270}{xi}{$\xx_\ii$}
\WPoint(20,5,y)\nput{315}{y}{$\yy$}
\WPoint(25,10,xi+1)\nput{90}{xi+1}{$\xx_{\ii +1}$}
\NPoint(30,5,C)
\AArrow(A,xi-1)
\AArrow(xi,xi-1)
\AArrow(xi,y)
\AArrow(y,xi+1)
\AArrow(C,xi+1)
\PArrow(xi-1,B)
\PArrow(y,B)
\end{picture}
\begin{picture}(35,15)(0,-6)
\psset{nodesep=0.7mm}
\psset{unit=1.2mm}
\put(-5,22){case $\ii \ge 2$}
\put(-3,18){with $\xx_{\ii + 1} < \xx_\ii$:}
\NPoint(0,5,A)
\WPoint(5,0,xi-1)\nput{225}{xi-1}{$\xx_{\ii - 1}$}
\WPoint(10,-5,B)
\WPoint(15,10,xi)\nput{90}{xi}{$\xx_\ii$}
\WPoint(20,5,y)\nput{45}{y}{$\yy$}
\WPoint(25,0,xi+1)\nput{270}{xi+1}{$\xx_{\ii +1}$}
\NPoint(30,5,C)
\AArrow(A,xi-1)
\AArrow(xi-1,xi)
\AArrow(y,xi)
\AArrow(xi+1,y)
\AArrow(xi+1,C)
\PArrow(B,xi-1)
\PArrow(B,y)
\end{picture}
\caption{\sf Zigzag reducibility: finding a vertex~$\yy$ between~$\xx_\ii$ and~$\xx_{\ii + 1}$ that admits a common upper/lower bound with~$\xx_{\ii - 1}$.}
\label{F:PosRed}
\end{figure}
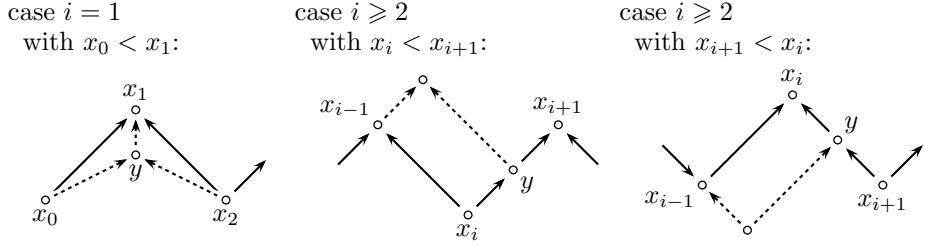

Comparing Definitions~\ref{D:Redpm} and~\ref{D:PosRed} immediately gives:

\begin{lemm}\label{L:Dict2}
For every local lattice~$\PP$, a zigzag~$\xxv$ on~$\PP$ is reducible if and only if the multifraction~$\FF(\xxv)$ is reducible.
\end{lemm}

Putting things together, we obtain a sufficient condition for the semi-convergence of reduction: 

\begin{prop}\label{P:PosNC}
If $\PP$ is a finite local lattice and every simple closed zigzag in~$\PP$ is reducible, then reduction is semi-convergent for~$\IM\PP$. 
\end{prop}

\begin{proof}
Let~$\aav$ be a $\piece$-minimal element of~$\Sim\MM$ (thus necessarily nontrivial). By Lemma~\ref{L:Dict}, there exists a closed zigzag~$\xxv$ in~$\PP$ satisfying~$\FF(\xxv) = \aav$. Let $\xxv'$ be a shortest closed subsequence of~$\xxv$ containing two vertices at least, which exists since $\xxv$ itself contains at least two vertices. Then $\xxv'$ is a simple closed zigzag. Hence, by assumption, it is reducible, and, therefore, so is~$\xxv$. By Lemma~\ref{L:Dict2}, it follows that $\aav$ is reducible. By Corollary~\ref{C:Base}, this implies that reduction is semi-convergent for~$\IM\PP$. 
\end{proof}

Note that, if $\PP$ is a finite poset, then the family of all simple closed zigzags on~$\PP$ is finite, since a zigzag longer than the cardinality of~$\PP$ cannot be simple.

Restricting to $\nn$-zigzags amounts to restricting to $\nn$-multifractions, and we obtain the following local version of Proposition~\ref{P:PosNC}:

\begin{coro}\label{C:PosNCn}
If $\PP$ is a finite local lattice and every simple closed $\pp$-zigzag in~$\PP$ with $\pp \le \nn$ is reducible, then reduction is $\nn$-semi-convergent for~$\IM\PP$. 
\end{coro}

\begin{rema}
By Lemma~\ref{L:Dict}, every simple multifraction in a monoid~$\IM\PP$ arises from a closed zigzag of~$\PP$. But, conversely, a closed zigzag in~$\PP$ need not induce a unital multifraction: typically, if $\PP$ is a bowtie $\{\xx_1 \wdots \xx_4\}$ with $\xx_1, \xx_3 \le \xx_2, \xx_4$, the monoid~$\IM\PP$ is free, and the multifraction $[\xx_1, \xx_2] / [\xx_3, \xx_2] / [\xx_3, \xx_4] / [\xx_1, \xx_4]$ associated with the closed zigzag $(\xx_1,\xx_2,\xx_3,\xx_4,\xx_1)$ is not unital. In fact, the correspondence is one-to-one if and only if the poset~$\PP$ is simply connected, meaning that every loop in~$\PP$ is homotopic to a point, with homotopy defined as adding or removing a pattern $(\xx, \yy, \xx)$ or interchanging $(\xx, \yy, \zz)$ and~$(\xx, \zz)$ for $\xx < \yy < \zz$~(see, for example, \cite{Rotman}). When this condition holds, the sufficient condition of Proposition~\ref{P:PosNC} is also necessary and, as a consequence, the semi-convergence of reduction for~$\IM\PP$ is a decidable property in case $\PP$ is finite and simply connected.
\end{rema}

\section{Examples and counter-examples}\label{S:Ex}

The criteria of Proposition~\ref{P:PosNC} and Corollary~\ref{C:PosNCn} enable us to construct explicit examples for various possible behaviours of multifraction reduction. Here we successively describe examples of monoids~$\MM$ for which reduction is semi-convergent, and for which it is $\pp$-semi-convergent for $\pp < \nn$ but not $\nn$-semi-convergent.

\subsection{Sufficient conditions}

Owing to Proposition~\ref{P:PosNC}, in order to obtain monoids for which reduction is semi-convergent, it suffices to find local lattices in which all simple closed zigzags are reducible. The question may be difficult in general, but finding sufficient conditions is easy. We begin with a general observation.

\begin{lemm}\label{L:Red3}
For every poset~$\PP$, and for $\nn \le 3$, every simple closed $\nn$-zigzag is reducible.
\end{lemm}

\noindent\begin{minipage}{\linewidth}
\begin{proof}
\rightskip40mm\VR(4,0) By definition, a $1$-zigzag is never simple. If $(\xx_0, \xx_1, \xx_2)$ is a simple closed $2$-zigzag, we have $\xx_1 \not= \xx_0 =\nobreak \xx_2$, and taking $\yy = \xx_0$ witnesses for reducibility. Finally, assume that $(\xx_0 \wdots \xx_3)$ is a simple closed $3$-zigzag. Assume for instance $\xx_0 < \xx_1$. Then taking $\yy = \xx_3$ witnesses for reducibility at~$1$, and taking $\yy = \xx_1$ witnesses for reducibility at~$2$, see the right hand side diagram. The picture is symmetric for~$\xx_0 > \xx_1$.\hfill
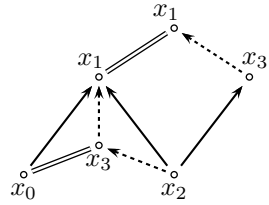
\begin{picture}(0,0)(-22,-8)
\def\NPoint(#1,#2,#3){\cnode[style=thin,fillcolor=white,linecolor=white](#1,#2){0}{#3}}
\def\WPoint(#1,#2,#3){\cnode[style=thin,fillcolor=white,fillstyle=solid](#1,#2){0.5}{#3}}
\def\AArrow(#1,#2){\ncline[nodesep=1mm, linewidth=0.8pt, border=2pt]{->}{#1}{#2}}
\def\PArrow(#1,#2){\ncline[linestyle=dashed,linewidth=0.8pt, border=1.2pt]{->}{#1}{#2}}
\def\DArrow(#1,#2){\ncline[style=double]{#1}{#2}}
\psset{nodesep=0.7mm}
\psset{xunit=1mm,yunit=1.3mm}
\WPoint(0,0,x0)\nput{270}{x0}{$\xx_0$}
\WPoint(10,3,y)\nput{270}{y}{$\xx_3$}
\WPoint(10,10,x1)\nput{110}{x1}{$\xx_1$}
\WPoint(20,0,x2)\nput{270}{x2}{$\xx_2$}
\WPoint(30,10,x3)\nput{75}{x3}{$\xx_3$}
\WPoint(20,15,z)\nput{110}{z}{$\xx_1$}
\AArrow(x0,x1)
\AArrow(x2,x1)
\AArrow(x2,x3)
\DArrow(x0,y)
\PArrow(x2,y)
\PArrow(y,x1)
\DArrow(x1,z)
\PArrow(x3,z)
\end{picture}
\end{proof}
\end{minipage}

\medskip So only simple closed zigzags of length~$4$ and above need to be considered. The semi-convergence question may be complicated in general. However, simple conditions turn out to be sufficient.

\begin{lemm}\label{L:SuffNC1}
Assume that $\PP$ is a finite local lattice satisfying
\begin{equation}\label{E:SuffNC1}
\parbox{115mm}{Any two elements with a common upper bound have a common lower bound.} 
\end{equation}
Then reduction is semi-convergent for~$\IM\PP$. 
\end{lemm}

\begin{proof}
We claim that every simple closed zigzag on~$\PP$ is reducible. Indeed, assume that $\xxv$ is a simple closed $\nn$-zigzag on~$\PP$. For $\nn \le 3$, $\xxv$ is reducible by Lemma~\ref{L:Red3}. Assume $\nn \ge 4$. Then there exists~$\ii$ with $2 \le \ii < \nn$ such that we have $\xx_{\ii - 1} < \xx_\ii$ and $\xx_{\ii + 1}Ê< \xx_\ii$, namely $\ii = 3$ in case $\xx_0 < \xx_1$, and $\ii = 2$ in case $\xx_0 > \xx_1$. By~\eqref{E:SuffNC1}, $\xx_{\ii - 1}$ and~$\xx_{\ii + 1}$, which admit the common upper bound~$\xx_\ii$, admit a common lower bound~$\yy$, and the latter witnesses for reducibility at~$\ii$. We conclude using Proposition~\ref{P:PosNC}.
\end{proof}

\begin{lemm}\label{L:SuffNC2}
Assume that $\PP$ is a finite local lattice satisfying
\begin{equation}\label{E:SuffNC2}
\parbox{115mm}{Every simple closed zigzag~$\xxv$ in~$\PP$ admits an interpolation center, meaning a vertex~$\yy$ such that $\xx_\ii \le \yy \le \xx_\jj$ holds for all~$\xx_\ii, \xx_\jj$ satisfying $\xx_\ii < \xx_\jj$.} 
\end{equation}
Then reduction is semi-convergent for~$\IM\PP$.
\end{lemm}

\begin{proof}
Assume that $\xxv$ is a positive simple closed $\nn$-zigzag with $\nn \ge 4$. If $\yy \neq \xx_1$, then~$\yy$ witnesses for reducibility at~$1$ since we have $\xx_0 \le \yy \le \xx_1$ and $\xx_2 \le \yy$. Assume $\yy = \xx_1$. Then $\xx_2 < \xx_3$ implies $\xx_2 \le \yy \le \xx_3$, and therefore $\xx_1$ and~$\xx_3$ admit a common upper bound, which implies reducibility at~$2$. The argument is symmetric for a negative zigzag. We conclude using Proposition~\ref{P:PosNC} again.
\end{proof}

\noindent\begin{minipage}{\textwidth}
\rightskip40mm\hspace{3.5mm}
For $\nn = 4$, the existence of an interpolation center corresponds to the usual interpolation property~$(IP)$.
We observed in Remark~\ref{R:Interpol} that, in every gcd-monoid, the left divisibility relation satisfies~$(IP)$. Saying that $\PP$ itself satisfied~$(IP)$ is a stronger hypothesis: $(IP)$ for divisibility in~$\IM\PP$ amounts to restricting to diagrams where $\xx_1$ and~$\xx_3$ admit a common lower bound.\hfill
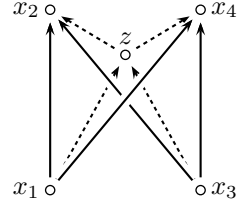
\begin{picture}(0,0)(-12,-2)
\psset{nodesep=1.5mm}
\psset{yunit=1.2mm}
\def\PArrow(#1,#2){\ncline[linestyle=dashed, nodesep=1mm, linewidth=0.8pt]{->}{#1}{#2}}
\WPoint(0,0,x1)\nput{180}{x1}{$\xx_1$}
\WPoint(20,0,x2)\nput{0}{x2}{$\xx_3$}
\WPoint(0,20,y1)\nput{180}{y1}{$\xx_2$}
\WPoint(20,20,y2)\nput{0}{y2}{$\xx_4$}
\WPoint(10,15,z)\nput{90}{z}{$\zz$}
\PArrow(x1,z) \PArrow(x2,z) \PArrow(z,y1) \PArrow(z,y2)
\AArrow(x2,y1) \AArrow(x1,y1) \AArrow(x1,y2) \AArrow(x2,y2)
\end{picture}
\end{minipage}

\subsection{Examples where reduction is semi-convergent}

Finding explicit examples, and therefore establishing Proposition~A of the introduction, is then easy.

\begin{prop}\label{P:PropA}
Let $\PPA$ be the 7-element poset whose Hasse diagram is depicted in Figure~\ref{F:PropA}, and let $\MMA$ be the associated interval monoid. Then $\MMA$ is a noetherian gcd-monoid for which reduction is semi-convergent but not convergent.
\end{prop}

\begin{proof}
The poset $\PPA$ is a meet-semilattice and it is easy to check directly that each of the seven posets $\INF{\PPA}\xx$ is a join-semilattice. Hence $\MMA$ is a gcd-monoid. By Corollary~\ref{C:FiniteNoeth}, the monoid~$\MMA$ is noetherian, since the poset~$\PPA$ is finite. Now~\eqref{E:SuffNC1} is clear, since~$0$ is a common lower bound for all elements of~$\PPA$. Let~$\MMA$ be the interval monoid of~$\PPA$. A presentation of~$\MMA$ is
\begin{equation}\label{E:P9}
\MON{\tta, \tta', \tta'', \ttb, \ttb', \ttb'', \ttc, \ttc', \ttc''}{\tta\ttb' = \ttb\tta'', \ttb\ttc' = \ttc\ttb'', \ttc\tta' = \tta\ttc''}.
\end{equation}
Then, by Lemma~\ref{L:SuffNC1}, reduction is semi-convergent for~$\MMA$. 

On the other hand, $\MMA$ does not satisfy the $3$-Ore condition: $\tta$, $\ttb$, and $\ttc$ pairwise admit common right multiples, but they admit no global right lcm. Therefore, by~\cite[Prop.~2.24]{Dit}, reduction cannot be convergent for~$\MMA$. 
\end{proof}

\begin{figure}[htb]
\begin{picture}(50,35)(0,-1)
\psset{nodesep=0.7mm}
\psset{xunit=0.9mm,yunit=1mm}
\WPoint(20,0,0)\nput{315}{0}{$\scriptstyle0$}
\WPoint(0,15,1)\nput{180}{1}{$\scriptstyle1$}
\WPoint(0,30,2)\nput{90}{2}{$\scriptstyle2$}
\WPoint(20,15,3)\nput{90}{3}{$\scriptstyle3$}
\WPoint(40,30,4)\nput{90}{4}{$\scriptstyle4$}
\WPoint(40,15,5)\nput{0}{5}{$\scriptstyle5$}
\WPoint(20,30,6)\nput{90}{6}{$\scriptstyle6$}
\AArrow(5,6)
\AArrow(1,6)
\AArrow(0,1)
\AArrow(0,3)
\AArrow(0,5)
\AArrow(1,2)
\AArrow(3,2)
\AArrow(3,4)
\AArrow(5,4)
\end{picture}
\begin{picture}(30,33)(0,-1)
\psset{nodesep=0.7mm}
\psset{xunit=1.5mm,yunit=1.7mm}
\WPoint(10,10,0)\nput{45}{0}{$\scriptstyle0$}
\WPoint(10,20,1)\nput{90}{1}{$\scriptstyle1$}
\WPoint(0,15,2)\nput{135}{2}{$\scriptstyle2$}
\WPoint(0,5,3)\nput{225}{3}{$\scriptstyle3$}
\WPoint(10,0,4)\nput{270}{4}{$\scriptstyle4$}
\WPoint(20,5,5)\nput{315}{5}{$\scriptstyle5$}
\WPoint(20,15,6)\nput{45}{6}{$\scriptstyle6$}
\AArrow(0,1)\naput{$\tta$}
\AArrow(0,3)\naput{$\ttb$}
\AArrow(0,5)\naput{$\ttc$}
\AArrow(1,2)\nbput{$\ttb'$}
\AArrow(3,2)\naput{$\tta''$}
\AArrow(3,4)\nbput{$\ttc'$}
\AArrow(5,4)\naput{$\ttb''$}
\AArrow(5,6)\nbput{$\tta'$}
\AArrow(1,6)\naput{$\ttc''$}
\end{picture}
\caption{\sf Two views of the Hasse diagram of the local lattice~$\PPA$ that satisfies Condition~\eqref{E:SuffNC1}, hence such that reduction is semi-convergent for~$\IM\PPA$.}
\label{F:PropA}
\end{figure}
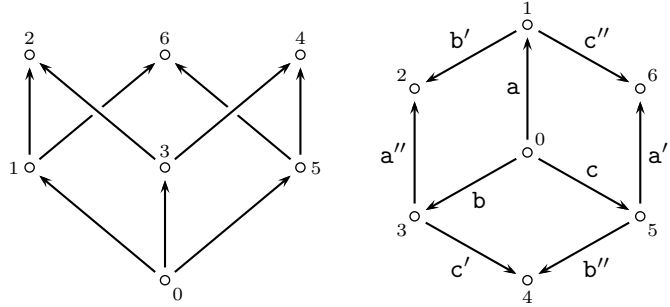

It may be observed that the enveloping group of~$\IM{\PPA}$ is a free group based (for instance) on $\{\tta, \tta', \ttb, \ttb', \ttc, \ttc'\}$, since the relations of~\eqref{E:P9} define the redundant generators~$\tta'', \ttb''$, and~$\ttc''$.

\begin{exam}
The truncated $4$-cube~$\PPB$ of Proposition~\ref{P:PropB} and Figure~\ref{F:PropB} is a local lattice, and it also satisfies~\eqref{E:SuffNC1}: two pairs admit a common upper bound in~$(\Pw(\Omega) \setminus \{\emptyset, \Omega\}, \subseteq)$ if and only if their intersection is nonempty, in which case they also admit a common lower bound (note that $\PPB$ consists of four glued copies of the poset~$\PPA$ of Proposition~\ref{P:PropA}). Similar examples can be obtained by considering the restriction of inclusion to any set $\{\XX \in \Pw(\Omega) \mid 1 \le \card\XX \le 3\}$ with~$\card\Omega \ge 4$.
\end{exam}

 The monoid~$\MMA$ of Proposition~\ref{P:PropA} provides an example for which reduction is semi-convergent but not convergent. However, in the case of~$\MMA$, semi-convergence holds because \eqref{E:SuffNC1} is true and there are very few simple zigzags in~$\PPA$. Starting from~$\PPA$, one can construct an infinite series of posets, all leading to semi-convergence without convergence, and admitting simple closed zigzags of arbitrary large size.

\begin{prop}\label{P:PropAn}
For $\nn \ge 1$, let $\PPAA\nn$ be the $3\nn + 4$~element poset obtained by alternatively gluing copies of~$\PPA$ and its mirror-image as in Figure~\ref{F:PropAn}, and let $\MMAA\nn$ be its interval monoid. Then $\MMAA\nn$ is a noetherian gcd-monoid for which reduction is semi-convergent but not convergent.
\end{prop}

\begin{proof}[Proof (sketch)]
The argument is similar to the one explained below for Proposition~\ref{P:PropC}. Comparing Figures~\ref{F:PropAn} (right) and~\ref{F:PropC} (right) shows that $\PPAA3$ is the full subgraph of~$\PPCC6$ obtained by erasing five adjacent diamonds out of twelve (and shifting the names of the vertices). As a consequence, there exists no simple closed zigzag in~$\PPAA3$ cycling around a central vertex as in~$\PPCC6$, and, by the argument of Proposition~\ref{P:PropC}, all other simple closed zigzags must be reducible. Therefore, reduction must be semi-convergent for~$\MMAA\nn$ with $\nn \le 3$. The result remains valid for~$\MMAA\nn$ with $\nn \ge 4$, as extending~$\PPAA\nn$ to~$\PPAA{\nn+1}$ creates new simple closed zigzags, but all eligible for the same reducibility argument. Note that the closed zigzag stemming from the peripheral circuit in the graph of~$\PPAA\nn$ is simple and has length~$2\nn + 2$ for $\nn \ge 2$. 

On the other hand, the failure of the $3$-Ore condition in~$\MMA$ remains valid in~$\MMAA\nn$ for every~$\nn$, so reduction cannot be convergent.
\end{proof}

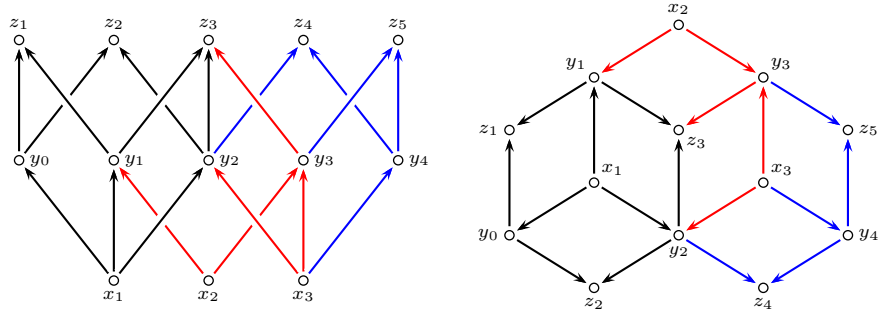
\begin{figure}[htb]
\def\AArrow(#1,#2){\ncline[nodesep=0.5mm, linewidth=0.8pt, border=2pt]{->}{#1}{#2}}
\def\BArrow(#1,#2){\ncline[linecolor=blue, nodesep=0.5mm, linewidth=0.8pt, border=2pt]{->}{#1}{#2}}
\def\CArrow(#1,#2){\ncline[linecolor=red, nodesep=0.5mm, linewidth=0.8pt, border=2pt]{->}{#1}{#2}}
\begin{picture}(62,35)(2,-2)
\psset{nodesep=0.8mm}
\psset{xunit=1.05mm,yunit=0.8mm}
\WPoint(12,0,x1)\nput{270}{x1}{$\scriptstyle\xx_1$}
\WPoint(24,0,x2)\nput{270}{x2}{$\scriptstyle\xx_2$}
\WPoint(36,0,x3)\nput{270}{x3}{$\scriptstyle\xx_3$}
\WPoint(0,20,y0)\nput{0}{y0}{$\scriptstyle\yy_0$}
\WPoint(12,20,y1)\nput{0}{y1}{$\scriptstyle\yy_1$}
\WPoint(24,20,y2)\nput{0}{y2}{$\scriptstyle\yy_2$}
\WPoint(36,20,y3)\nput{0}{y3}{$\scriptstyle\yy_3$}
\WPoint(48,20,y4)\nput{0}{y4}{$\scriptstyle\yy_4$}
\WPoint(0,40,z1)\nput{90}{z1}{$\scriptstyle\zz_1$}
\WPoint(12,40,z2)\nput{90}{z2}{$\scriptstyle\zz_2$}
\WPoint(24,40,z3)\nput{90}{z3}{$\scriptstyle\zz_3$}
\WPoint(36,40,z4)\nput{90}{z4}{$\scriptstyle\zz_4$}
\WPoint(48,40,z5)\nput{90}{z5}{$\scriptstyle\zz_5$}
\CArrow(x2,y1) \BArrow(y2,z4) \BArrow(y4,z4) 
\AArrow(y0,z2) \AArrow(y2,z2) \AArrow(y0,z1) \AArrow(y1,z1) \AArrow(y1,z3) \AArrow(y2,z3) 
\AArrow(x1,y0) \AArrow(x1,y1) \AArrow(x1,y2) 
\CArrow(y3,z3) \CArrow(x2,y3) \CArrow(x3,y2) \CArrow(x3,y3)
\BArrow(y3,z5) \BArrow(y4,z5) \BArrow(x3,y4) 
\end{picture}
\begin{picture}(45,32)(0,-1)
\psset{nodesep=0.5mm}
\psset{xunit=0.75mm,yunit=0.7mm}
\WPoint(15,20,x1)\nput{45}{x1}{$\scriptstyle\xx_1$}
\WPoint(30,50,x2)\nput{90}{x2}{$\scriptstyle\xx_2$}
\WPoint(45,20,x3)\nput{45}{x3}{$\scriptstyle\xx_3$}
\WPoint(0,10,y0)\nput{180}{y0}{$\scriptstyle\yy_0$}
\WPoint(15,40,y1)\nput{135}{y1}{$\scriptstyle\yy_1$}
\WPoint(30,10,y2)\nput{270}{y2}{$\scriptstyle\yy_2$}
\WPoint(45,40,y3)\nput{45}{y3}{$\scriptstyle\yy_3$}
\WPoint(60,10,y4)\nput{0}{y4}{$\scriptstyle\yy_4$}
\WPoint(0,30,z1)\nput{180}{z1}{$\scriptstyle\zz_1$}
\WPoint(15,0,z2)\nput{270}{z2}{$\scriptstyle\zz_2$}
\WPoint(30,30,z3)\nput{315}{z3}{$\scriptstyle\zz_3$}
\WPoint(45,00,z4)\nput{270}{z4}{$\scriptstyle\zz_4$}
\WPoint(60,30,z5)\nput{0}{z5}{$\scriptstyle\zz_5$}
\AArrow(y0,z2) \AArrow(y2,z2) \AArrow(y0,z1) \AArrow(y1,z1) \AArrow(y1,z3) \AArrow(y2,z3) 
\AArrow(x1,y0) \AArrow(x1,y1) \AArrow(x1,y2) 
\CArrow(y3,z3) \CArrow(x2,y1) \CArrow(x2,y3) \CArrow(x3,y2) \CArrow(x3,y3)
\BArrow(y2,z4) \BArrow(y4,z4) \BArrow(y3,z5) \BArrow(y4,z5) \BArrow(x3,y4) 
\end{picture}
\caption{\sf The Hasse diagram of the poset $\PPAA\nn$, for $\nn = 3$, again in two different forms, emphasizing the levels of vertices (left), and as a planar graph (right): $\PPAA1$ (in black) coincides with~$\PPA$, and $\PPAA{\nn + 1}$ is obtained from~$\PPAA\nn$ by adding 3~vertices and 5~edges (in red for~$\PPAA2$, in blue for~$\PPAA3$), so as to form a new copy of~$\PPA$ or its reversed image.}
\label{F:PropAn}
\end{figure}

The monoid~$\MMAA2$ witnesses for another property. By Proposition~\ref{P:PropAn}, reduction is semi-convergent for~$\MMAA2$, so, in other words, the latter satisfies Conjecture~$\ConjA$ of~\cite{Diu}. However, in~$\MMAA2$, let $\aav$ be the $6$-multifraction
$$[\yy_0, \zz_1]/[\xx_2, \zz_1]/[\xx_2, \zz_5]/[\yy_4, \zz_5]/[\yy_4, \zz_4] / [\yy_2, \zz_4] / [\yy_2, \zz_2] / [\yy_0, \zz_2],$$
corresponding to the outer boundary in Figure~\ref{F:PropAn} (right). Then $\aav$ is unital, and it reduces to~$\one$ as semi-convergence requires. But applying to~$\aav$ maximal tame reduction~\cite[Prop.~4.18]{Diu} at the successive levels $1, 2, 3, 4, 5, 1, 2, 3, 1$ yields the multifraction $1/[\xx_1, \yy_0]/1/1/ [\xx_1, \yy_0] / 1$, which is not trivial. This shows that the universal recipe of~\cite[Def.~4.27]{Diu} may fail in~$\MMAA2$ and, therefore, the latter does not satisfy Conjecture~$\ConjB$ of~\cite{Diu}. Thus, Conjecture~$\ConjB$ may be strictly stronger than Conjecture~$\ConjA$. 

\subsection{Examples where reduction is not semi-convergent}

In the other direction, we now describe examples where reduction is not semi-convergent; more precisely, we establish Proposition~C of the introduction by constructing for every even integer~$\nn$ a monoid for which reduction is $\pp$-semi-convergent for~$\pp < \nn$, but not $\nn$-semi-convergent. So, we look for posets in which some closed zigzags are irreducible, the point being to ensure that the local lattice condition is true and all short zigzags are reducible.

\begin{prop}\label{P:PropC}
For $\nn \ge 4$ even, let $\PPCC\nn$ be the poset with domain $\{\yy\} \cup \{\xx_\ii, \yy_\ii, \zz_\ii \mid \ii = 1 \wdots \nn\} $ and relations $\xx_\ii \le \yy_\ii \le \zz_\ii$ and $\xx_\ii \le \yy_{\ii - 1} \le \zz_\ii$ for $\ii = 1 \wdots \nn$ with $\yy_0 = \yy_\nn$, and $\xx_\ii \le \yy \le \zz_\jj$ for $\ii$ odd and $\jj$ even, and let $\MMCC\nn$ be its interval monoid. Then $\MMCC\nn$ is a noetherian gcd-monoid for which reduction is $\pp$-semi-convergent for~$\pp < \nn$ but not $\nn$-semi-convergent.
\end{prop}

\begin{figure}[htb]
\def\AArrow(#1,#2){\ncline[nodesep=0.3mm, linewidth=0.8pt, border=2pt]{->}{#1}{#2}}
\def\CArrow(#1,#2){\ncline[linecolor=red, nodesep=0.3mm, linewidth=0.8pt, border=2pt]{->}{#1}{#2}}
\def\PArrow(#1,#2){\ncline[linestyle=dashed, nodesep=0.3mm, linewidth=0.8pt, border=0.8pt]{->}{#1}{#2}}
\begin{picture}(65,49)(0,-3)
\psset{nodesep=0.5mm}
\psset{unit=1.05mm}
\WPoint(25,20,0)\nput{0}{0}{$\scriptstyle\yy$}
\WPoint(10,0,1)\nput{270}{1}{$\scriptstyle\xx_1$}
\WPoint(20,-1,2)\nput{270}{2}{$\scriptstyle\xx_2$}
\WPoint(40,0,3)\nput{270}{3}{$\scriptstyle\xx_3$}
\WPoint(30,1,4)\nput{270}{4}{$\scriptstyle\xx_4$}
\WPoint(15,19,5)\nput{180}{5}{$\scriptstyle\yy_1$}
\WPoint(35,19,6)\nput{180}{6}{$\scriptstyle\yy_2$}
\WPoint(50,20,7)\nput{0}{7}{$\scriptstyle\yy_3$}
\WPoint(0,20,8)\nput{180}{8}{$\scriptstyle\yy_4$}
\WPoint(30,39,9)\nput{90}{9}{$\scriptstyle\zz_2$}
\WPoint(40,40,10)\nput{90}{10}{$\scriptstyle\zz_3$}
\WPoint(20,41,11)\nput{90}{11}{$\scriptstyle\zz_4$}
\WPoint(10,40,12)\nput{90}{12}{$\scriptstyle\zz_1$}
\PArrow(4,8) \PArrow(4,7) \PArrow(8,11) \PArrow(7,11) 
\CArrow(2,0) \CArrow(4,0) \CArrow(0,12) \CArrow(0,10) 
\AArrow(1,5) \AArrow(1,8) \AArrow(2,5) \AArrow(2,6) \AArrow(3,6) \AArrow(3,7) \AArrow(5,12) \AArrow(5,9) \AArrow(8,12) \AArrow(6,9) \AArrow(6,10) \AArrow(7,10) 
\end{picture}
\begin{picture}(45,43)(-3,-3)
\psset{nodesep=0.5mm}
\psset{unit=1.05mm}
\WPoint(20,20,0)
\WPoint(20,0,1)\nput{270}{1}{$\scriptstyle\xx_1$}
\WPoint(10,20,2)\nput{180}{2}{$\scriptstyle\xx_2$}
\WPoint(20,40,3)\nput{90}{3}{$\scriptstyle\xx_3$}
\WPoint(30,20,4)\nput{0}{4}{$\scriptstyle\xx_4$}
\WPoint(10,10,5)\nput{225}{5}{$\scriptstyle\yy_1$}
\WPoint(10,30,6)\nput{135}{6}{$\scriptstyle\yy_2$}
\WPoint(30,30,7)\nput{45}{7}{$\scriptstyle\yy_3$}
\WPoint(30,10,8)\nput{315}{8}{$\scriptstyle\yy_4$}
\WPoint(0,20,9)\nput{180}{9}{$\scriptstyle\zz_2$}
\WPoint(20,30,10)\nput{90}{10}{$\scriptstyle\zz_3$}
\WPoint(40,20,11)\nput{0}{11}{$\scriptstyle\zz_4$}
\WPoint(20,10,12)\nput{270}{12}{$\scriptstyle\zz_1$}
\AArrow(4,8) \AArrow(4,7) \AArrow(8,11) \AArrow(7,11) 
\CArrow(2,0) \CArrow(4,0) \CArrow(0,12) \CArrow(0,10) 
\AArrow(1,5) \AArrow(1,8) \AArrow(2,5) \AArrow(2,6) \AArrow(3,6) \AArrow(3,7) \AArrow(5,12) \AArrow(5,9) \AArrow(8,12) \AArrow(6,9) \AArrow(6,10) \AArrow(7,10) 
\nput{45}{0}{$\scriptstyle\yy$}
\end{picture}

\begin{picture}(65,60)(5,-7)
\psset{nodesep=0.5mm}
\psset{xunit=1.05mm,yunit=1.1mm}
\WPoint(30,20,0)\nput{0}{0}{$\scriptstyle\yy$}
\WPoint(45,2,5)\nput{180}{5}{$\scriptstyle\xx_4$}
\WPoint(25,0,6)\nput{180}{6}{$\scriptstyle\xx_5$}
\WPoint(5,0,1)\nput{270}{1}{$\scriptstyle\xx_6$}
\WPoint(0,20,7)\nput{180}{7}{$\scriptstyle\yy_6$}
\WPoint(60,20,10)\nput{0}{10}{$\scriptstyle\yy_3$}
\WPoint(47,22,11)\nput{0}{11}{$\scriptstyle\yy_4$}
\WPoint(13,22,12)\nput{180}{12}{$\scriptstyle\yy_5$}
\WPoint(55,40,16)\nput{90}{16}{$\scriptstyle\zz_4$}
\WPoint(35,42,17)\nput{90}{17}{$\scriptstyle\zz_5$}
\WPoint(15,42,18)\nput{90}{18}{$\scriptstyle\zz_6$}
\PArrow(5,10) \PArrow(5,11) \PArrow(10,16) \PArrow(11,16)
\PArrow(6,11) \PArrow(6,12) \PArrow(11,17) \PArrow(12,17) 
\PArrow(1,12) \PArrow(1,7) \PArrow(12,18) \PArrow(7,18) 
\WPoint(15,-2,2)\nput{270}{2}{$\scriptstyle\xx_1$}
\WPoint(35,-2,3)\nput{270}{3}{$\scriptstyle\xx_2$}
\WPoint(55,2,4)\nput{270}{4}{$\scriptstyle\xx_3$}
\WPoint(20,18,8)\nput{0}{8}{$\scriptstyle\yy_1$}
\WPoint(40,18,9)\nput{180}{9}{$\scriptstyle\yy_2$}
\WPoint(5,40,13)\nput{90}{13}{$\scriptstyle\zz_1$}
\WPoint(25,38,14)\nput{90}{14}{$\scriptstyle\zz_2$}
\WPoint(45,38,15)\nput{90}{15}{$\scriptstyle\zz_3$}
\CArrow(1,0) \CArrow(3,0) \CArrow(5,0) \CArrow(0,13) \CArrow(0,15) \CArrow(0,17) 
\AArrow(2,7) \AArrow(2,8) \AArrow(7,13) \AArrow(8,13) 
\AArrow(3,8) \AArrow(3,9) \AArrow(8,14) \AArrow(9,14) 
\AArrow(4,9) \AArrow(4,10) \AArrow(9,15) \AArrow(10,15) 
\end{picture}
\begin{picture}(45,56)(0,-1)
\psset{nodesep=0.5mm}
\psset{xunit=0.8mm,yunit=0.7mm}
\WPoint(30,40,0)
\WPoint(45,30,1)\nput{315}{1}{$\scriptstyle\xx_6$}
\WPoint(30,0,2)\nput{270}{2}{$\scriptstyle\xx_1$}
\WPoint(15,30,3)\nput{225}{3}{$\scriptstyle\xx_2$}
\WPoint(0,60,4)\nput{135}{4}{$\scriptstyle\xx_3$}
\WPoint(30,60,5)\nput{90}{5}{$\scriptstyle\xx_4$}
\WPoint(60,60,6)\nput{45}{6}{$\scriptstyle\xx_5$}
\WPoint(45,10,7)\nput{315}{7}{$\scriptstyle\yy_6$}
\WPoint(15,10,8)\nput{225}{8}{$\scriptstyle\yy_1$}
\WPoint(0,40,9)\nput{135}{9}{$\scriptstyle\yy_2$}
\WPoint(15,70,10)\nput{135}{10}{$\scriptstyle\yy_3$}
\WPoint(45,70,11)\nput{45}{11}{$\scriptstyle\yy_4$}
\WPoint(60,40,12)\nput{0}{12}{$\scriptstyle\yy_5$}
\WPoint(30,20,13)\nput{270}{13}{$\scriptstyle\zz_1$}
\WPoint(0,20,14)\nput{225}{14}{$\scriptstyle\zz_2$}
\WPoint(15,50,15)\nput{135}{15}{$\scriptstyle\zz_3$}
\WPoint(30,80,16)\nput{90}{16}{$\scriptstyle\zz_4$}
\WPoint(45,50,17)\nput{45}{17}{$\scriptstyle\zz_5$}
\WPoint(60,20,18)\nput{315}{18}{$\scriptstyle\zz_6$}
\AArrow(5,10) \AArrow(5,11) 
\AArrow(6,11) \AArrow(6,12) 
\AArrow(7,18) \AArrow(12,18) \AArrow(11,17) 
\CArrow(1,0) \CArrow(3,0) \CArrow(5,0) \CArrow(0,13) \CArrow(0,15) \CArrow(0,17) 
\AArrow(1,12) \AArrow(1,7) 
\AArrow(2,7) \AArrow(2,8) 
\AArrow(3,8) \AArrow(3,9) 
\AArrow(4,9) \AArrow(4,10) 
\AArrow(7,13) \AArrow(11,16) \AArrow(12,17) 
\AArrow(8,13) \AArrow(8,14) 
\AArrow(9,14) \AArrow(9,15) 
\AArrow(10,15) \AArrow(10,16) 
\nput{0}{0}{$\ \scriptstyle\yy$}
\end{picture}
\caption{\sf The Hasse diagram of the poset $\PPCC\nn$, for $\nn = 4$ (top) and $\nn = 6$ (bottom), viewed as a necklace of $\nn$ connected diamonds plus a central $\nn$-ray cross connecting each other endpoint (left) and as a planar graph (right).}
\label{F:PropC}
\end{figure}
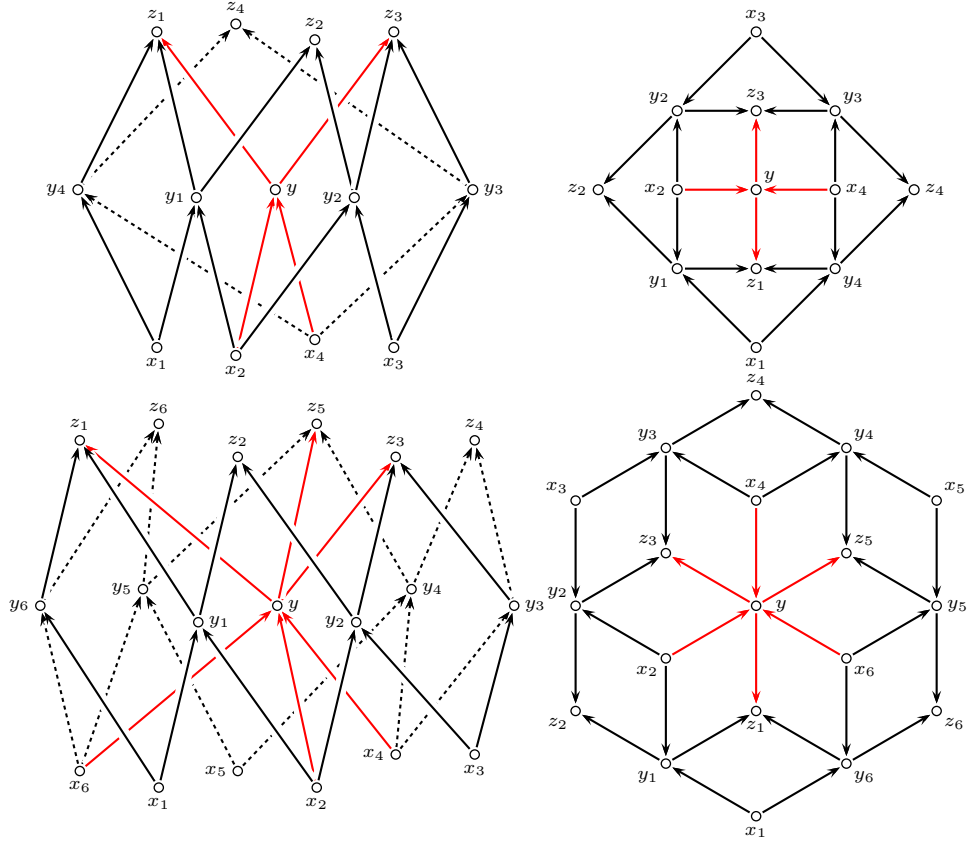

\begin{proof}
The local lattice condition is checked directly by considering the various types of vertices in~$\PPCC\nn$: typically, for $\ii$ odd, $\SUP{\PPCC\nn}{\xx_\ii}$ is a 4-element lattice whereas, for $\ii$ even, it consists of~$\xx_\ii$ plus the six elements $\yy_\ii$, $\yy_{\ii - 1}$, $\yy$, $\zz_{\ii - 1}$, $\zz_\ii$, $\zz_{\ii + 1}$ and is a copy of the meet-semilattice~$\PPA$ of Proposition~\ref{P:PropA}. Hence $\MMCC\nn$ is a gcd-monoid. Moreover, as $\PPCC\nn$ is finite, $\MMCC\nn$ is noetherian.

As $\MMCC\nn$ is an interval monoid, it embeds into its group by Proposition~\ref{P:NF} and, therefore, reduction is $2$-semi-convergent.

Consider the $\nn$-multifraction 
$$\aav = [\xx_1, \zz_2] / [\xx_3, \zz_2] / [\xx_3, \zz_4] \sdots [\xx_{\nn - 1}, \zz_\nn] / [\xx_1, \zz_\nn],$$
which corresponds to the exterior loop in the right hand side diagram for~$\PPCC\nn$. As the diagram is tiled by squares that correspond to defining relations of the monoid, $\aav$ is unital. We claim that it is irreducible. Indeed, assume $\ii$ even, hence negative in~$\aav$. The only nontrivial left divisor of~$[\xx_{\ii +1}, \zz_{\ii + 2}]$ in~$\MMCC\nn$ is $[\xx_{\ii + 1}, \yy_{\ii + 1}]$. Now $[\xx_{\ii + 1}, \yy_{\ii + 1}]$ and $[\xx_{\ii + 1}, \zz_\ii]$ admit no common right multiple in~$\MMCC\nn$: this can be obtained directly from \cite[Prop.\,5.10]{Weh}, but a self-contained argument runs as follows. The only decomposition of~$[\xx_{\ii + 1}, \zz_\ii]$ is $[\xx_{\ii + 1}, \yy_\ii] \opp [\yy_\ii, \zz_\ii]$. By definition, $[\xx_{\ii + 1}, \yy_{\ii + 1}]$ and $[\xx_{\ii + 1}, \yy_\ii]$ admit a right lcm, which is $[\xx_{\ii + 1}, \yy_\ii] \opp [\yy _\ii, \zz_{\ii + 1}]$. Now $[\yy_\ii, \zz_\ii]$ and $[\yy _\ii, \zz_{\ii + 1}]$ admit no common right multiple in~$\MMCC\nn$, since $\zz_\ii$ and~$\zz_{\ii + 1}$ admit no common upper bound in~$\PPCC\nn$. Hence reduction is not $\nn$-semi-convergent for~$\MMCC\nn$. 

Now, we claim that, for $\pp < \nn$ even, all simple closed $\pp$-zigzags in~$\PPCC\nn$ are reducible. Indeed, assume that $\xxv$ is a simple closed zigzag of length~$\pp <\nobreak \nn$. Let $\gamma$ be a loop in the Hasse diagram of~$\PPCC\nn$ connecting the points of~$\xxv$ (such a path is not unique: in each diamond, one can choose one side or the other). We claim that, if $\gamma$ does not visit the central vertex~$\yy$, then $\xxv$ must be reducible: indeed, $\pp < \nn$ implies that $\gamma$ is too short to circle around~$\yy$ (in the sense of the right hand side diagrams in Figure~\ref{F:PropC}) and, therefore, it must contain a U-turn that corresponds either to a pattern $\ss, \tt, \ss$ (one arrow crossed back and forth) or to a pattern $\ss, \tt, \uu, \ss$ (four arrows around a diamond), both directly implying that $\xxv$ is reducible. Now assume that $\gamma$ visits~$\yy$, and consider what happens in the preceding steps: owing to the symmetries, we can assume with loss of generality that $\gamma$ reaches~$\yy$ from~$\zz_1$, which must be an entry of~$\xxv$, since it is maximal in~$\PPCC\nn$. Before~$\zz_1$, $\gamma$ can come either from~$\yy$, in which case $\xxv$ is reducible since $\gamma$ includes the pattern $\yy, \zz_1, \yy$, or from~$\yy_1$, or from~$\yy_\nn$, the latter two cases being symmetric. So assume $\gamma$ contains~$\yy_1, \zz_1, \yy$. Before that, $\gamma$ can come either from~$\zz_1$, in which case $\xxv$ is obviousy reducible, or from~$\zz_2$ (downwards), or from~$\xx_1$ or~$\xx_2$ (upwards). If $\gamma$ contains~$\zz_2, \yy_1, \zz_1, \yy$, then $\xxv$, which then contains the entries $\zz_2, \yy_1, \zz_1, \xx$ for some $\xx \le \yy < \zz_1$, is reducible, because $\yy$ and~$\yy_1$ admit the common lower bound~$\xx_2$. If $\gamma$ contains~$\xx_1, \yy_1, \zz_1, \yy$, the vertex before~$\xx_1$ must be either~$\yy_1$, in which case $\xxv$ is reducible, or~$\yy_\nn$, in which case $\xxv$ is also reducible as we have $\yy_\nn < \zz_1$. Finally, if $\gamma$ contains~$\xx_2, \yy_1, \zz_1, \yy$, then $\xxv$ is reducible, because we have $\xx_2 < \yy$. By applying Corollary~\ref{C:PosNCn}, we conclude that reduction is $\pp$-semi-convergent for~$\MMCC\nn$. (For $\nn \ge 6$, in the special case of $4$-zigzags, we can alternatively observe that the vertices~$\yy_\ii$ and~$\yy$ witness that the poset~$\PPCC\nn$ has the interpolation property and apply Lemma~\ref{L:SuffNC2}.)
\end{proof}

\begin{rema}\label{R:DivExp}
Say that $(\cc_1 \wdots \cc_\nn)$ is a \emph{central cross} for an $\nn$-multifraction~$\bbv$ if $\bb_\ii = \cc_{\ii - 1} \cc_\ii$ (\resp $\bb_\ii = \cc_\ii \cc_{\ii -1}$) holds for every $\ii$ positive (\resp negative) in~$\bbv$ (with the convention $\cc_0 = \cc_\nn$). The right hand side diagrams in Figure~\ref{F:PropC} show that the multifraction
$$\bbv = [\xx_\nn, \zz_1] / [\xx_2, \zz_1] / [\xx_2, \zz_3] \sdots [\xx_\nn, \zz_{\nn - 1}]$$
admits a central cross and that the multifraction~$\aav$ witnessing for the failure of $\nn$-semi-convergence is an lcm-expansion of~$\bbv$, meaning~\cite{Dit} that, for each~$\ii \le \nn$, there are decompositions $\aa_\ii = \aa'_\ii\aa''_\ii$, $\bb_\ii = \bb'_\ii \bb''_\ii$ satisfying $\aa'_{\ii - 1} \bb''_\ii = \aa'_ \ii \bb''_{\ii + 1} = \bb''_\ii \lcmt \bb''_{\ii + 1}$ (\resp $\bb'_\ii \aa''_{\ii - 1} = \bb'_{\ii + 1} \aa''_ \ii = \bb'_\ii \lcm \bb'_{\ii + 1}$) for $\ii$ positive (\resp negative) in~$\bbv$. Multifractions admitting a central cross are, in some sense, the simplest unital multifractions, and they are always reducible; their lcm-expansions, which are also unital, appear as the next complexity step in the family of unital multifractions. Thus, if we say that reduction is \emph{weakly $\nn$-semi-convergent} for~$\MM$ if every lcm-expansion of an $\nn$-multifraction with a central cross is either trivial or reducible, then reduction fails not only to be $\nn$-semi-convergent, but even to be weakly $\nn$-semi-convergent for~$\MMCC\nn$. 
\end{rema}

\begin{rema}
The monoid~$\MMCC\nn$ contradicts the known alternative forms of semi-convergence. For instance, the element $\gg = [\xx_1, \zz_2][\xx_3, \zz_2]\inv$ in~$\EG{\MMCC4}$ is represented by two distinct irreducible fractions, namely $[\xx_1, \zz_2] / [\xx_3, \zz_2]$ and $[\xx_1, \zz_4] / [\xx_3, \zz_4]$. On the other hand, the $6$-multifraction 
$$[\xx_1, \yy_1] / [\xx_2, \yy_1] / [\xx_2, \zz_3] / [\xx_4, \zz_3] / [\xx_4, \yy_4] / [\xx_1, \yy_4]$$ 
turns out to be reducible both to~$\one$ and to $\aav \opp \one$, where $\aav$ is the irreducible multifraction in the proof of Proposition~\ref{P:PropC}, contradicting weak confluence in~$\MMCC4$.
\end{rema}

As in Section~\ref{S:Embed}, some quotients of the monoids~$\MMCC\nn$ share their properties but have fewer atoms. Starting from~$\MMCC4$, which has $20$~atoms, one finds that
$$\MON{\tta \wdots \mathtt{f, x, y}}{\mathtt{ab=ba,cd=dc,ef=fe,db=x^\mathrm2,eb=y^\mathrm2,ca=xy,fa=yx}}$$
is a gcd-monoid~$\MM$ for which reduction fails to be $4$-semi-convergent (and even weakly $4$-semi-convergent): the unital $4$-multifraction $\mathtt{ac/bd/af/be}$ is irreducible. Similarly, starting from~$\MMCC6$, which has $30$~atoms, we find that 
$$\begin{matrix}
\langle\mathtt{a \wdots f, x, y, z} \mid \mathtt{ab=ba,cd=dc,ef=fe,}\hspace{4cm}\\ 
\hspace{2cm}\mathtt{ea=xy,ae=yx,db=zy,bd=yz,fc=xz,cf=zx}\rangle^+,
\end{matrix}$$
is a gcd-monoid in which $\mathtt{ac/ed/fb/ca/de/bf}$ is unital and irreducible, contradicting~$6$-semi-convergence---and even weak $6$-semi-convergence---whereas $$\mathtt{ac/ed/f/a/b/c/de/bf}$$ reduces both to~$\one$ and to the previous nontrivial multifraction. However, it is not clear that, in the above quotients, $2$-semi-convergence is preserved, that is, that these monoids embed into their respective groups.

\section{Extending the method}\label{S:Ext}

We briefly discuss further extensions of the previous results.

\subsection{Category monoids}\label{SS:CatMon}

To any poset~$P$ one can associate a (small) category $\Cat{P}$, whose objects are the elements of~$P$ and where there is an arrow from~$x$ to~$y$, which is then unique, if and only if $x\le y$ holds. Now to every category~$\CCC$, one can associate its \emph{universal monoid} $\Um{\CCC}$, defined by the generators $\overline{f}$, where~$f$ is an arrow of~$\CCC$, and the relations
 \begin{align}
 \overline{fg}&=\overline{f}\cdot\overline{g}\,,
 &&\text{whenever }fg\text{ is defined},\label{E:CatMon1}\\
 \overline{f}&=1\,,
 &&\text{whenever }f\text{ is an identity of }\CCC\,.
 \label{E:CatMon2}
 \end{align}
Equivalently, viewing a monoid as a category with exactly one object, $\Um{\CCC}$, together with the canonical functor (morphism of categories) $\CCC\to\Um{\CCC}$, $f\mapsto\overline{f}$, is an initial object in the category of all functors from~$\CCC$ to a monoid.

In~\cite{Weh}, many of the results proved here for the construction $P\mapsto\IM{P}$ are established in the more general context of universal monoids of categories.
Let us consider, for example, Proposition~\ref{P:NF}.
Part~(ii) of that result, about the existence of a unique normal form, can be extended to the universal monoid of any category: this is contained in \cite[Lemma~3.4]{Weh}, and can ultimately be traced back to Higgins~\cite{Higg1971}.
On the other hand, Part~(i) of Proposition~\ref{P:NF}, which states the embeddability of~$\IM{P}$ into its group, cannot be extended to an arbitrary category, simply because there are monoids that cannot be embedded into any group (consider non-cancellative monoids!).

Nevertheless, it is proved in \cite[Thm.\,10.1]{Weh} that, for any category~$\CCC$, the mon\-oid~$\Um{\CCC}$ embeds into its group if and only if there are a group~$G$ and a functor $\varphi\colon\CCC\to G$ such that the restriction of~$\varphi$ to every hom-set of~$\CCC$ is one-to-one. This is applied, in \cite[Ex.\,10.2]{Weh}, to the example below, leading us to a new gcd-monoid for which reduction is semi-convergent but not convergent.

\begin{prop}\label{P:PropC6}
Let 
\begin{equation}\label{E:C6}
\MMD = \MON{\tta, \ttb, \ttc, \tta', \ttb', \ttc'}{\mathtt{ab'=ba'\,,\ bc'=cb'\,,\ ac'=ca'}}.
\end{equation}
Then reduction is semi-convergent but not convergent for~$\MMD$. 
\end{prop}

\begin{proof}[Proof (sketch)]
By using the above-mentioned \cite[Thm.\,10.1]{Weh}, it is proved in~\cite{Weh} that~$\MMD$ embeds into its group, implying that reduction is $2$-semi-convergent for~$\MMD$. 

To prove that reduction is semi-convergent for~$\MMD$, one observes, as in \cite[Ex.\,10.2]{Weh}, that~$\MMD$ is the universal monoid of the finite category~$\CCC_6$ with three objects~$0$, $1$, $2$, and arrows~$\tta$, $\ttb$, $\ttc$ from~$0$ to~$1$ and~$\tta'$, $\ttb'$, $\ttc'$ from~$1$ to~$2$:
$$
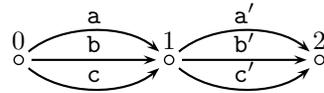
\begin{picture}(40,10)(0,0)
\WPoint(0,4,0)\nput{90}{0}{$0$}
\WPoint(20,4,1)\nput{90}{1}{$1$}
\WPoint(40,4,2)\nput{90}{2}{$2$}
\ncarc[nodesep=1mm,arcangle=40]{->}01\naput{$\tta$}
\ncarc[nodesep=1mm,arcangle=0]{->}01\naput{$\ttb$}
\ncarc[nodesep=1mm,arcangle=-40]{->}01\naput{$\ttc$}
\ncarc[nodesep=1mm,arcangle=40]{->}12\naput{$\tta'$}
\ncarc[nodesep=1mm,arcangle=0]{->}12\naput{$\ttb'$}
\ncarc[nodesep=1mm,arcangle=-40]{->}12\naput{$\ttc'$}
\end{picture}$$
and the relations~\eqref{E:C6} satisfied.
Zigzags need to be replaced by finite composable sequences~$\xxv$ of non-identity arrows of either~$\CCC_6$ or its opposite category, with the source and the target of~$\xxv$ identical.
The crucial point is the observation that Proposition~\ref{P:PosBase} can be extended, with a similar proof, to any category.

Finally, reduction is not convergent for~$\MMD$, as the $3$-Ore condition fails in the monoid~$\MMD$: the elements~$\tta, \ttb, \ttc$ pairwise admit common right multiples, but they admit no global common right multiple.
\end{proof}

\subsection{Artin-Tits monoids}\label{SS:AT}

Multifraction reduction was primarily introduced for investigating Artin-Tits monoids, which differ from the interval monoids of posets considered above in many aspects. However, the method of proof developed in Section~\ref{S:NCInt} might extend to further monoids. Owing to Proposition~\ref{P:Base}, the point would be to identify a family that contains all $\piece$-minimal unital multifractions and study their reduction. The approach may be of any interest only if that family is significantly smaller than the family of all multifractions---as is the family of simple multifractions in the case of an interval monoid.

What makes interval monoids (and, more generally, the category monoids of Subsection~\ref{SS:CatMon}) specific is the existence of the source and target maps~$\sr, \tg$, and the possibility of using them to split the elements of the monoid and the multifractions. As mentioned in Remark~\ref{R:Garside}, the normal decomposition of Proposition~\ref{P:NF} is a special case of the greedy normal form associated with a Garside family, and, in this case, simple multifractions are obtained by gluing matching elements of the smallest Garside family. A smallest Garside family exists in every gcd-monoid, and imitating the construction of Section~\ref{S:NCInt} suggests to call a multifraction simple if its entries lie in the smallest Garside family and are matching, in some sense to be defined. In the case of an Artin-Tits of spherical type, or more generally of FC~type, the notion of a signed word drawn in a finite fragment of the Cayley graph of~$\MM$ as considered in~\cite[Def.~V.2.2]{Dhr} could provide a natural candidate. However, in contrast with the case of interval monoids, the family of $\piece$-minimal multifractions is infinite in general: for instance, if $\MM$ is the free commutative monoid on~$\{\tta, \ttb, \ttc\}$ (which satisfies the $3$-Ore condition, and even the $2$-Ore condition), the $6$-multifraction $\tta^\pp/\ttb^\pp/\ttc^\pp/\tta^\pp/\ttb^\pp/\ttc^\pp$ is unital and $\piece$-minimal for every~$\pp \ge 1$. Thus, even in such an easy case, describing all $\piece$-minimal multifractions is not obvious, and there seems to be still a long way before completing the approach for an arbitrary Artin--Tits monoid.

\bibliographystyle{plain}

\newcommand{\noopsort}[1]{}

\end{document}